\documentclass[a4paper,12pt]{amsart}
\usepackage{amsfonts}
\usepackage{amsmath,array}
\usepackage{amssymb}
\usepackage{mathrsfs}
\usepackage[colorlinks]{hyperref}
 \usepackage{graphicx}
\usepackage{enumerate}
\usepackage[usenames]{color}
\usepackage{cases}
\usepackage{yhmath}
\title[Trigonometric bases in noncommutative $L_p(\mathbb{T}^d_\theta)$]{Trigonometric bases in noncommutative $L_p(\mathbb{T}^d_\theta)$ spaces and associated partial sum operators}
\author{B. Ozbekbay}
\address{
Al-Farabi Kazakh National University, 050040 Almaty, Kazakhstan and
Institute of Mathematics and Mathematical Modeling, 050010 Almaty, Kazakhstan.}
\email{ozbekbay.b00@gmail.com}

\author{F. Sukochev}
\address{School of Mathematics and Statistics, University of New South
Wales, Kensington, 2052, NSW, Australia.}
\email{f.sukochev@unsw.edu.au}
\author{K. Tulenov}
\address{Institute of Mathematics and Mathematical Modeling, 050010 Almaty, Kazakhstan and School of Mathematics and Statistics, University of New South Wales, Kensington, NSW, 2052, Australia}
\email{kanat.tulenov@unsw.edu.au}
\date{}

\newtheorem{theorem}{Theorem}[section]
\newtheorem{lemma}[theorem]{Lemma}
\newtheorem{definition}[theorem]{Definition}
\newtheorem{corollary}[theorem]{Corollary}
\newtheorem{rem}[theorem]{Remark}
\newtheorem{proposition}[theorem]{Proposition}
\newtheorem{convention}[theorem]{Convention}
\newtheorem{example}[theorem]{Example}
\begin{document}

\date{}
\begin{abstract}  We develop a harmonic-analytic method for constructing a generalized trigonometric system in noncommutative $L_p(\mathbb{T}^d_\theta)$ spaces arising from the strongly continuous representation of $\mathbb{T}^d$ and show that the generalized trigonometric system is a Schauder basis in $L_p(\mathbb{T}^d_\theta)$ for $1<p<\infty.$ In particular, we prove that this trigonometric system forms an RUC-basis in $L_p(\mathbb{T}^d_\theta)$ for $2<p<\infty.$
Our results provide a noncommutative counterpart of the classical trigonometric basis in $L_p(\mathbb{T}^d)$. Further, we obtain a weak $(1,1)$ type estimate of partial sum operators associated with noncommutative trigonometric systems. This allows us to study uniformly boundedness of partial sum operators     between pairs of symmetric spaces that do not necessarily possess nontrivial Boyd indices, extending known results in \cite{DDdPS, DdPS, R, R2} to the setting of quasi-Banach symmetric spaces.
\end{abstract}
\maketitle

\section{Introduction}
The notion of a Schauder bases originates from the seminal work of Schauder in 1927 \cite{S1927}, 
where infinite-dimensional Banach spaces were first equipped with coordinate-type representations 
analogous to those in finite-dimensional linear algebra. A sequence $\{x_n\}_{n\ge1}$ in a Banach 
space $X$ is a Schauder basis if every $x\in X$ admits a unique expansion
$$
x=\sum\limits_{n=1}^\infty a_n x_n, \quad a_n \in \mathbb{C},
$$
with convergence in the norm of $X$. Comprehensive treatments of Schauder bases, including unconditionality, and applications to classical Banach function spaces, are given in the two monographs of Lindenstrauss and Tzafriri \cite{LT} (see also \cite{KS, KSa}). These books provide a definitive account of various bases in the classical Lebesgue $L_p$-spaces, emphasizing the central role played by unconditional bases for $1<p<\infty$. 
Within this framework, the trigonometric system occupies a distinguished position. Classical results due to Paley and Zygmund \cite{PZ1932}, Marcinkiewicz \cite{M1936} and Zygmund \cite{MZ1937} show that for $1<p<\infty$ 
the trigonometric system forms a Schauder basis in $L_p(\mathbb T^d)$. This fact underlies a large part of Fourier analysis and summability theory, allowing many 
operator-theoretic questions to be reduced to scalar inequalities for the Fourier coefficients.  
Their contributions remain fundamental for understanding when and how Schauder bases exist and behave in classical settings.

In contrast, while noncommutative $L_p$-spaces 
associated with von Neumann algebras have been widely studied, only a few works \cite{DDdPS3,FS1,FS2,FS3,JNRX2004} 
address the existence and structure of Schauder-type decompositions. 
%{\color{red}General existence results for completely bounded bases in noncommutative $L_p$-spaces, based on the local structure of $\mathscr{COL}_p$-spaces, were obtained in \cite{JNRX2004}.}
Noncommutative torus $\mathbb{T}^d_{\theta}$ (also known as irrational rotation algebra) is a landmark example in 
non-commutative geometry  (see \cite{Connes1980}, \cite[Chapter 12]{green-book} and \cite{CM2014})  and in noncommutative harmonic analysis \cite{CXY,  LMSZ, McDSX, Ri, Rieffel, Spera, Tulenov1, STT, XXY}.  Recently, there has been substantial  body of work on generalising the methods of classical harmonic analysis on the torus to the noncommutative torus.
The noncommutative measure space $(\mathbb{T}^d_{\theta},\tau_{\theta})$ is equipped with a finite normal trace $\tau_{\theta}$ on the von Neumann algebra $\mathbb{T}^d_{\theta},$ which generalizes the Lebesgue integral (recovering it when $\theta = 0$). This allows the construction of noncommutative $L_p$-spaces, denoted $L_p(\mathbb{T}^d_\theta) := L_p(\mathbb{T}^d_\theta, \tau_{\theta})$. When $\theta = 0$, these spaces reduce to the classical $L^p$-spaces on the torus $\mathbb{T}^d.$ To our knowledge, no reference 
explicitly establishes that the trigonometric (Fourier unitary) system forms a Schauder basis of 
$L_p(\mathbb{T}_\theta^d)$, the noncommutative torus, despite its frequent implicit use. 
The present paper fills this gap by proving that a suitable enumeration of the Fourier unitaries
$$
\{e^{\theta}_{\gamma}=U_1^{\gamma_1}\cdots U_d^{\gamma_d}\}_{\gamma\in\mathbb Z^d},
$$
where $U_j$, $j=1,\dots,d$, are generating unitaries of the noncommutative torus $\mathbb T_\theta^d$,
which play the role of the classical exponentials 
$\{e^{2\pi i \gamma t}\}_{\gamma\in\mathbb Z^d}$ on $\mathbb T^d$,
yields a Schauder basis in $L_p(\mathbb T_\theta^d)$ for $1<p<\infty.$
Let $\{r_k\}_{k=1}^{\infty}$ be the usual Rademacher sequence given by
$$
r_k(t) = \operatorname{sign}(\sin(2^{k}\pi t)), \quad 0 \le t \le 1, \quad k\in \mathbb N,
$$
and let $X$ be a Banach space with a Schauder basis $\{x_n\}_{n\ge1}.$ The sequence $\{x_n\}_{n=1}^{\infty}$ is called an RUC-basis if for every
$x = \sum\limits_{n=1}^{\infty} \alpha_n x_n \in X,$
the series
$$
\sum\limits_{n=1}^{\infty} r_n(t)\,\alpha_n x_n
$$
converges for almost all $t \in [0,1]$ (see \cite{DS1997}). It is well known \cite[Chapter~1]{KSa} that if $p\neq2$, then
$L_p[0,1]$ admits no uniformly bounded orthonormal unconditional basis.
Motivated by this phenomenon, we prove that the trigonometric system $\{e^{\theta}_{\gamma}\}_{\gamma\in\mathbb Z^d}$ forms an RUC-bases in noncommutative $L_p(\mathbb{T}^d_\theta)$ spaces for $2<p<\infty$.
This notion has its origin in the general theory of orthonormal systems in function
spaces, where it was established long time ago (see \cite[p.~47]{O1975}) that if an
orthonormal system $\{e_n\}_{n\ge1}$ with
$$
\sup_{n\geq 1}\|e_n\|_{L_p[0,1]}<\infty, \quad 2\leq p<\infty,
$$
forms a Schauder basis of $L_p[0,1]$, then it is also an RUC-basis in $L_p[0,1]$. 
A classical illustration is given by the
trigonometric and Walsh--Paley systems \cite{BKPS}. Taken in their natural ordering,
these systems form bases of $L_p[0,1]$, $1<p<\infty$, which are unconditional
if and only if $p=2$ \cite[Theorem 1.d.6]{LT}. On the other hand, they form
RUC-bases in $L_p[0,1]$ for all $2<p<\infty$
\cite[Remark~5]{BKPS}, see also \cite{SWS}.
Further examples arise in noncommutative settings. Pisier showed
\cite[Theorems 3.1 and~3.7]{BKPS} that if $2<p<\infty$, then the family of natural matrix units $\{e_{mn}\}_{m,n\geq 1}$, taken in the rectangular ordering, forms an RUC-basis 
in the Schatten classes $C_p$ for $2<p<\infty$ (see also \cite[Theorem 5(a)]{LP}). 
 While
unconditionality always implies the RUC property, the trigonometric system
shows that RUC systems may exist far beyond the unconditional framework
(see \cite[Corollary~1.4 and Remark~V following Corollary~2.2]{BKPS}).
Further progress on RUC-bases in symmetric spaces of functions and noncommutative settings were made in \cite{ACT, DSS2002, DS1997, S1994, S1998}.

Most existing results on Schauder bases in symmetric function and operator spaces are obtained under the assumption that the underlying spaces have nontrivial Boyd indices. It is well known that, the trigonometric system forms a Schauder basis in a symmetric space if and only if the space has nontrivial Boyd indices \cite{DDdPS3, FS1,FS2,FS3, LT}, a condition that guarantees the uniform boundedness of the associated basis projections via interpolation methods. 
In contrast, the more general situation, where the underlying symmetric spaces may have trivial Boyd indices has been less explored. In this paper, we investigate this general framework by studying the mapping properties of partial sum operators and Hilbert transforms between pairs of symmetric spaces that do not necessarily admit nontrivial Boyd indices. Our results extend those obtained in \cite{DDdPS, DdPS, R, R2} to the broader setting of quasi-Banach symmetric spaces.

As a first step, we establish a uniform weak-type $(1,1)$ estimate for the partial sums associated with the corresponding noncommutative trigonometric systems. We then analyze the relationship between these partial sum operators and the Hilbert transform $\mathcal H_{\mathbb Z^d_+}$ (see \eqref{HA}) associated with the positive cone $\mathbb Z^d_+$. This operator $H_{\mathbb Z^d_+}$ may be realized as a particular instance of the Hilbert transform introduced in \cite[Definition~4.1]{R2}, corresponding to the compact abelian group $G=\mathbb T^d$ and the von Neumann algebra $\mathcal M=\mathbb T^d_\theta$.

Consequently, the boundedness of $\mathcal H_{\mathbb Z^d_+}$ on $L_p(\mathbb T^d_\theta)$ for $1<p<\infty$, together with its weak-type $(1,1)$ estimate, follows from \cite[Theorem~4.2]{R2} and \cite[Corollary~4.6]{R2}. When combined with \cite[Theorem~14]{STZ1}, these results enable us to study the boundedness of the Hilbert transform and the uniform boundedness of the partial sums of the associated noncommutative Schauder basis acting between noncommutative symmetric spaces $\mathcal E(\mathbb T^d_\theta)$ and $\mathcal F(\mathbb T^d_\theta)$, thereby extending and unifying earlier results in this direction \cite{ABG, BGM, DDdPS, R, R2, Z} in the special case $\mathcal{M}=\mathbb{T}_\theta^d$.

\section{Preliminaries}

\subsection{Symmetric (quasi-)Banach function spaces}

Let $S(0,1)$ denote the space of all measurable real-valued functions on the interval $(0,1)$ with the Lebesgue measure $m$ and identify functions that agree almost everywhere. For a function $f \in S(0,1)$, its decreasing rearrangement $\mu(f)$ is defined by
$$
\mu(t,f) = \inf\{ s \geq 0 : m(\{|f| > s\}) \leq t \}, \quad t > 0.
$$

For $1\leq p \leq \infty$, let $L_{p}(0,1)$ represent the classical Lebesgue space of $p$-integrable functions (or essentially bounded functions in the case $p=\infty$). The Lorentz space $L_{p,q}(0,1)$, for $1\leq  p, q \leq \infty$, consists of all measurable functions $f\in S(0,1)$ for which the following quasi-norm is finite:
\begin{equation*}
    \|f\|_{L_{p,q}(0,1)}=\left\{ \begin{array}{rcl}
         \left(\int\limits_{0}^1\left(t^{\frac{1}{p}}\mu(t,f)\right)^q\frac{dt}{t}\right)^{\frac{1}{q}}, & \mbox{for}
         & q<\infty, \\ \sup\limits_{t>0}t^\frac{1}{p}\mu(t,f),\;\;\;\;\;\;\;\;\;\;\;\; & \mbox{for} & q=\infty,  
                \end{array}\right. 
\end{equation*}
For a comprehensive treatment of these spaces, we refer the reader to \cite{G2008}.

\begin{definition}\label{Sym}
A space $(E(0,1),\|\cdot\|_{E(0,1)})$ is called a symmetric (quasi-)Banach function space on $(0,1)$ if it satisfies the following axioms:
\begin{enumerate}[{\rm (a)}]
    \item $E(0,1)$ is a linear subspace of $S(0,1)$;
    \item $(E(0,1),\|\cdot\|_{E(0,1)})$ forms a complete (quasi-)normed space;
    \item For any $f \in E(0,1)$ and $g \in S(0,1)$ with $\mu(g) \leq \mu(f)$, it follows that $g \in E(0,1)$ and $\|g\|_{E(0,1)} \leq \|f\|_{E(0,1)}$.
\end{enumerate}
\end{definition}

Classical examples of Banach and quasi-Banach symmetric spaces include the Lebesgue spaces $L_{p}(0,1)$ and the Lorentz spaces $L_{p,q}(0,1)$ for $1\leq  p, q \leq \infty$, respectively.

We say that a function $g\in S(0,1)$ is \emph{submajorized} by $f\in S(0,1)$ in the sense of Hardy--Littlewood--Pólya (denoted $\mu(g) \prec\prec \mu(f)$) if
$$
\int\limits_{0}^t \mu(s,g) \, ds \leq \int\limits_{0}^t \mu(s,f) \, ds, \,\ t \geq 0.
$$
For $s>0$ consider
the dilation operator $D_s$ on $S(0,1)$ defined as
$$(D_s (f) )(t)=\begin{cases} f(t/s),& \text{if}\;\; 0\leq t\leq \min\{1,s\},\\
0,& \text{otherwise}.\\
\end{cases}$$
The \textit{Boyd indices} of a symmetric Banach function space $ E(0,1) $ are defined by 
$$
\alpha_E = \liminf_{0 < s < 1} \frac{\log \|D_s\|_{E \to E}}{\log s} = \lim_{s \to 0^+} \frac{\log \|D_s\|_{E \to E}}{\log s}.
$$
$$
\beta_E = \limsup_{s > 1} \frac{\log \|D_s\|_{E \to E}}{\log s} = \lim_{s \to \infty} \frac{\log \|D_s\|_{E \to E}}{\log s},
$$
In general, $ 0 \le \alpha_E \le \beta_E \le 1 $. We say that the Boyd indices are \textit{non-trivial} if $ 0 < \alpha_E \le \beta_E < 1 $.
The general theory and definitions of symmetric (quasi-)Banach function spaces can be found in \cite{BS, KPS} and \cite{LT}.
\subsection{Noncommutative  torus}  For $d\geq 1$ we denote the $d$-torus by $\mathbb{T}^{d}$ defined by
$$
\mathbb{T}^{d} = \{(t_{1}, \ldots, t_{d}): |t_{j}| = 1,\ t_{j} \in \mathbb{C},\ 1 \leq j \leq d\}.
$$
The torus is equipped with the usual topology and multiplicative group law, i.e. for $t = (t_{1}, \ldots, t_{d}) \in \mathbb{T}^{d}$ and $w = (w_{1}, \ldots, w_{d}) \in \mathbb{T}^{d}$,
$$
t \cdot w = (t_{1}w_{1}, \ldots, t_{d}w_{d}).
$$
For $\gamma = (\gamma_{1}, \ldots, \gamma_{d}) \in \mathbb{Z}^{d}$ and $t = (t_{1}, \ldots, t_{d}) \in \mathbb{T}^{d}$, define
$$
t^{\gamma} = t_{1}^{\gamma_{1}} \cdots t_{d}^{\gamma_{d}}.
$$
Given $d \geq 2$ and let $\theta = (\theta_{jk})_{1 \leq j,k \leq d}$ be a real antisymmetric matrix.
Let $C(\mathbb{T}^d_{\theta})$ be a universal $*$-algebra generated by unitaries $\{U_k\}_{k=1}^d$ 
satisfying the conditions
$$
U_{j} U_{k} = e^{i\theta_{j k}} U_{k} U_{j}, \quad 1 \leq j, k \leq d.
$$
For $\gamma = (\gamma_{1}, \ldots, \gamma_{d}) \in \mathbb{Z}^{d},$ we define 
$$ 
e^{\theta}_{\gamma} = U_{1}^{\gamma_{1}} \cdots U_{d}^{\gamma_{d}}.
$$
Note that 
\begin{equation}\label{1}
    e_{\gamma}^\theta e_{\eta}^\theta = e^{-i \sum\limits_{j<k} \theta_{jk} {\eta}_j {\gamma}_k} e_{\gamma+\eta}^{\theta},
\quad
(e_{\gamma}^\theta)^* = e^{-i \sum\limits_{j<k} \theta_{jk} {\gamma}_j {\gamma}_k} e_{-\gamma}^{\theta}.
\end{equation}
A polynomial in $C(\mathbb{T}^{d}_{\theta})$ is a finite sum \begin{eqnarray*}x= \sum\limits_{\gamma\in\mathbb{Z}^d}\alpha(\gamma) e^\theta_{\gamma},\quad\alpha(\gamma)\in\mathbb{C},\quad \gamma\in\mathbb{Z}^d,    \end{eqnarray*}
that is, $\alpha(\gamma) = 0$ for all but finite indices $\gamma\in\mathbb{Z}^d.$ The involution algebra $\mathcal{P}_{\theta}$ of all such polynomials is dense in $C(\mathbb{T}^{d}_{\theta})$.
Define a linear functional $\tau_\theta : C(\mathbb{T}^d_{\theta}) \to \mathbb{C}$ by setting
$$
\tau_\theta (e^{\theta}_{\gamma}) = 0 \quad \text{unless } \gamma = (\gamma_1, \ldots, \gamma_d) = 0, \ \mathbf{\gamma} \in \mathbb{Z}^d,
$$
and
$$
\tau_\theta (e^{\theta}_{\gamma}) = 1 \quad \text{when } \gamma = (\gamma_1, \ldots, \gamma_d) = 0, \ \gamma \in \mathbb{Z}^d.
$$
It can be demonstrated that $\tau_\theta$ is positive, that is, 
$\tau_\theta (x^* x) \geq 0$ for $x \in C(\mathbb{T}^d_{\theta})$ (see, \cite[p. 190]{LMSZ}).
We equip the linear space $C(\mathbb{T}^d_{\theta})$ with an inner product defined by the formula
$$
\langle x, y \rangle = \tau_\theta (y^*x), \quad x, y \in C(\mathbb{T}^d_{\theta}).
$$
The action $\lambda$ of $C(\mathbb{T}^d_{\theta})$ on the pre-Hilbert space $(C(\mathbb{T}^d_{\theta}), \langle\cdot, \cdot\rangle)$ 
by left multiplication extends to a representation of $C(\mathbb{T}^d_{\theta})$ as bounded operators 
on the completed Hilbert space $L_2(\mathbb{T}_\theta^d)$.

The closure in the weak operator topology of $\lambda(C(\mathbb{T}^d_{\theta}))$ 
in $B(L_2(\mathbb{T}_\theta^d))$ is denoted by $\mathbb{T}_\theta^d$, 
and $\tau_\theta$ extends to a faithful normal tracial state on $\mathbb{T}_\theta^d.$ 

 This algebra $\mathbb{T}^{d}_{\theta},$ is referred to as the $d$-dimensional noncommutative (or quantum) torus. Note that the von Neumann algebra $\mathbb{T}^{d}_{\theta}$ is hyperfinite \cite[p. 759]{CXY}. For more details on the noncommutative torus, refer to \cite{CXY, LMSZ, McDSX, XXY}.

\subsection{Noncommutative spaces }
Let $B(H)$ denote the algebra of all bounded linear operators in a Hilbert space $H$. Let $\mathcal M \subset B(H)$ be a finite von Neumann algebra equipped with a faithful, normal, finite trace $\tau$. A closed, densely defined operator $x:{\rm dom}(x) \to H$ is said to be \emph{affiliated} with $\mathcal{M}$ if $xu=ux$ for each unitary operator $u$ in the commutant $\mathcal{M}'$ of $\mathcal{M}$.

An operator $x$ affiliated with $\mathcal{M}$ is called $\tau$\emph{-measurable} if for every $\varepsilon>0$ there exists a projection $P\in\mathcal{M}$ such that $P(H) \subset {\rm dom}(x)$ and $\tau(1-P)\leq \varepsilon.$ We denote by $S(\mathcal M)$ the set of all $\tau$-measurable operators.
For $x\in S(\mathcal{M})$, the distribution function $d(s,|x|)$ is defined by
$$
d(s,|x|)=\tau( \chi_{(s,\infty)}(|x|) ), \quad s\ge0,
$$
where $\chi_{(s,\infty)}(|x|)$ is the spectral projection of  $|x|=(x^*x)^{1/2}$ corresponding to the interval $(s,\infty)$.
For $x \in S(\mathcal M)$ the decreasing rearrangement (or 
generalized singular value function) of $x$ is the function 
$\mu(x) : t \mapsto \mu(t,x)$ defined by
$$
\mu(t,x) = \inf\{\, s > 0 : d(s,|x|)
   \leq t \,\}, \quad t\geq 0.
$$
Using the properties of the decreasing rearrangement $\mu(x)$ and its
right–continuous inverse, we have
\begin{equation}\label{d}
    d(s,|x|)=m\{t\ge0:\mu(t,x)>s\}, \quad s\ge0,
\end{equation}
where $m$ denotes the Lebesgue measure on $(0,\infty)$ (see \cite[Chapter~III, pp.~129–130]{DdPS}).

\begin{definition}\label{NC Sym}
A linear subset $\mathcal{E} \subset S(\mathcal{M})$ equipped with a complete quasi-norm $\|\cdot\|_{\mathcal{E}}$ is called a quasi-Banach symmetric operator space if, for every $x \in \mathcal{E}$ and every $y \in S(\mathcal{M})$ satisfying $\mu(y) \leq \mu(x)$, this implies $y \in \mathcal{E}$ and $\|y\|_{\mathcal{E}} \leq \|x\|_{\mathcal{E}}$.
\end{definition}
In the case where the von Neumann algebra is commutative, specifically $\mathcal{M}=L_{\infty}(0,1)$,
the construction of a symmetric operator space coincides with a symmetric function space
on $(0,1)$. The construction of noncommutative symmetric operator spaces proceeds as follows. Given a symmetric (quasi-)Banach function space $E$ on $(0,1)$, we define
$$
\mathcal{E}(\mathcal{M}) = \{ x \in S(\mathcal{M}) : \mu(x) \in E \}
$$
and endow it with the natural (quasi-)norm
$$
\|x\|_{\mathcal{E}(\mathcal{M})} := \|\mu(x)\|_E, \quad x \in \mathcal{E}(\mathcal{M}).
$$
Then $(\mathcal{E}(\mathcal{M}), \|\cdot\|_{\mathcal{E}(\mathcal{M})})$ forms a (quasi-)Banach space, called the \emph{noncommutative symmetric (quasi-)Banach operator space} associated with $(\mathcal{M},\tau)$ corresponding to $(E,\|\cdot\|_{E})$ \cite{KS, F}.

As established in \cite{DdPS, KS, F} (see also \cite[Section 3]{LSZ}), the mapping
$$
(E, \|\cdot\|_{E}) \longleftrightarrow (\mathcal{E}(\mathcal{M}), \|\cdot\|_{\mathcal{E}(\mathcal{M})})
$$
defines a one-to-one correspondence between symmetric operator spaces in $S(\mathcal{M})$ and symmetric function spaces in $S(0,1)$.
For further details on the theory of noncommutative (quasi-)Banach symmetric spaces over general semifinite von Neumann algebras, we refer to \cite{ DdPS, LSZ, PXu}.

The family $\{e^{\theta}_{\gamma}\}_{\gamma \in \mathbb{Z}^{d}}$ forms an orthonormal basis in $L_{2}(\mathbb{T}_{\theta}^{d})$ (see \cite{CXY,  Tulenov1}), satisfying the following properties:
\begin{equation}\label{orthogonal-relations}
\begin{split}
&\tau_{\theta}((e^{\theta}_{\gamma})^{*}e^{\theta}_{\eta}) = \delta_{\gamma,\eta}, \\
&\|e^{\theta}_{\gamma}\|_{L_{\infty}(\mathbb{T}^{d}_{\theta})} = \|(e^{\theta}_{\gamma})^{*}\|_{L_{\infty}(\mathbb{T}^{d}_{\theta})} = 1,
\end{split}
\end{equation}
where $\delta_{\gamma,\eta}$ denotes the Kronecker delta. Furthermore, the inclusion $L_{q}(\mathbb{T}^{d}_{\theta}) \subset L_{p}(\mathbb{T}^{d}_{\theta})$ is valid for all $1\leq  p \leq q \leq \infty$.

\begin{definition}\cite[Definition 1.a.1]{LT} 
A sequence $\{x_k\}_{k=1}^\infty$ in a Banach space $X$ is called a \emph{Schauder basis} of $X$ if for every $x\in X$ there exists a unique sequence of scalars $\{a_k\}_{k=1}^\infty\subset \mathbb C$ such that
$$
x=\sum\limits_{k=1}^\infty a_k x_k.
$$
%A sequence $\{x_n\}_{n=1}^\infty$ which is a Schauder basis of its closed linear span is called a basic sequence.
\end{definition}

\begin{proposition}\cite[Proposition 1.a.3]{LT}\label{Schauder}
    A sequence $ \{x_k\}_{k=1}^{\infty}$ in a Banach space $X$ is Schauder basis if and only if the following conditions are satisfied:
\begin{enumerate}
    \item $ x_k \neq 0 $ for every $ k \in \mathbb{N} $,
    \item $ [x_k]_{k=1}^{\infty} = X $ (where $[x_k]$ the closed linear span of $ \{x_k\} $ in $ X $),
    \item There exists a constant $ c(X)>0$, depending only on the space $ X $, such that for all natural numbers $ n_1 \leq n_2 $ and for all sequences $ \{a_k\}_{k=1}^{n_2} \subset \mathbb{C} $, the following inequality holds:
    $$
    \left\| \sum\limits_{k=1}^{n_1} a_k x_k \right\|_X \leq c(X) \left\| \sum\limits_{k=1}^{n_2} a_k x_k \right\|_X.
    $$
\end{enumerate}
\end{proposition}

 \begin{definition}\cite[Definition 1.2.15]{HNVW}
Let $X$ be a Banach space and let $1 \leq p < \infty$.  
The Lebesgue--Bochner space $L^p(\mathbb{T},X)$ consists of all (equivalence classes of) strongly measurable functions 
$$
f : \mathbb{T} \to X
$$
such that
$$
\|f\|_{L^p(\mathbb{T},X)} 
= \Big( \int\limits_{\mathbb{T}} \|f(t)\|_X^p \, dt \Big)^{1/p} < \infty,
$$
where $dt$ is the
normalised Haar measure on $\mathbb{T}$.  

%The $\ell_p$-space of $X$-valued sequences is defined by
%$$\ell_p(X) = \Big\{ (x_k)_{k \in \mathbb{Z}} \; : \; x_k \in X,\; \sum\limits_{k\in \mathbb{Z}}\| x_k \|_X^{p} < \infty \Big\},$$
%equipped with the norm
%$$\big\| (x_k) \big\|_{\ell_p(X)} = \Big( \sum\limits_{k=1}^{\infty} \| x_k \|_X^{p} \Big)^{1/p}.$$
\end{definition}  
For every $f \in L_p(\mathbb{T},X)$ and $\psi \in L_1(\mathbb{T})$ the function
\begin{equation}\label{*}
   (f * \psi )(s) = \int\limits_{\mathbb{T}} \psi (t)\,f(st^{-1})\,dt, \quad s \in \mathbb{T}, 
\end{equation}
where the integral is understood as the Bochner integral in $L_p(\mathbb{T},X)$.  
This function is called the \emph{convolution} of $f$ and $\psi $. Moreover, the following norm estimate holds:
\begin{equation}\label{Young}
    \|f * \psi \|_{L_p(\mathbb{T},X)} \leq \|f\|_{L_p(\mathbb{T},X)}\,\|\psi \|_{L_1(\mathbb{T})}, \, f\in L_p(\mathbb{T},X), \, \psi\in L_1(\mathbb{T}),
\end{equation}
see \cite[p.227]{P}  and \cite{CXY} for more details.

For a function $f\in L_p(\mathbb T, X)$, $1\le p<\infty$, the
$k$-th Fourier coefficient of $f$ is defined by
\begin{equation}\label{F-coef}
\widehat{f}(k)
    = \int\limits_{\mathbb{T}} f(t)\, t^{-k}\, dt,
    \quad k\in\mathbb Z,
\end{equation}
where the integral is understood in the Bochner sense.

\begin{definition} Let $X$ be a Banach space and $1 < p < \infty$.  We say that $X$ has the UMD property (is a UMD space) if there exists a constant $C_{p,X} > 0$ such that for every finite martingale difference sequence $\{d_k\}_{k=1}^n$ in $L_p(\mathbb{T}, X)$ and for every choice of signs $\varepsilon_k \in \{-1,1\}$ one has
$$
\Big\| \sum\limits_{k=1}^n \varepsilon_k d_k \Big\|_{L_p(\mathbb{T},X)}
\leq C_{p,X} \, \Big\| \sum\limits_{k=1}^n d_k \Big\|_{L_p(\mathbb{T},X)}.
$$  
\end{definition}
In particular, for each $1<p<\infty$, the noncommutative space $L_p(\mathbb T_\theta^d)$ has the UMD property (see \cite[Corollary~4.2]{DDdPS}).

\begin{definition}\cite[Definition 4.2.1]{HNVW}
Let $X$ be a Banach space and $ 1 < p < \infty $. The periodic Hilbert transform on $ L_p(\mathbb{T}, X) $ is defined by
\begin{equation}\label{H_X}
    (H_X f)(s) = \lim_{\varepsilon \to 0^+} \frac{1}{2\pi} \int\limits_{\varepsilon \le |t| \le \pi} \cot\left(\frac{t}{2}\right) f(s - t)  dt,
\end{equation}
for $ f \in L_p(\mathbb{T}, X) $, where the limit in the $ L_p(\mathbb{T}, X) $-norm and almost everywhere.
\end{definition}

\begin{theorem}\cite{B}\label{HUMD}
For a Banach space $ X $, the following are equivalent:
\begin{enumerate}
    \item $ X $ has the UMD property.
    \item The Hilbert transform $ H_X $ is a bounded linear operator on $ L_p(\mathbb{T}, X) $ for all $ p \in (1, \infty) $.
\end{enumerate}
\end{theorem}
For the periodic Hilbert transform $H_X$ on $L_p(\mathbb{T}, X)$, the Fourier coefficients in terms of \eqref{F-coef} are given by
\begin{equation}\label{HT}
\widehat{H_X f}(k) = -i\operatorname{sgn}(k) \hat{f}(k), \quad k \in \mathbb{Z}.
\end{equation}
%see \cite[p. 389]{HNVW}.

For each $t=(t_1,\dots,t_d)\in\mathbb T^d$, we define
\begin{equation}\label{Pit}
  \pi_t(e_{\gamma}^\theta)
 = t^{\gamma} e_{\gamma}^\theta
 = t_1^{\gamma_1}\cdots t_d^{\gamma_d} U_1^{\gamma_1}\cdots U_d^{\gamma_d},
 \quad \gamma \in\mathbb{Z}^d. 
\end{equation}
Then $\{\pi_t\}_{t\in\mathbb T^d}$ is a strongly continuous representation of $\mathbb T^d$
by trace–preserving $*$-automorphisms of $\mathbb T^d_\theta$ (see \cite[p.760]{CXY}, \cite[p.1237]{McDSX}). In particular, for every $0<p\le\infty$, $\pi_t$ extends to an isometry on $L_p(\mathbb T^d_\theta)$, so that
\begin{equation}\label{PI}
    \|\pi_t(x)\|_{L_p(\mathbb T^d_\theta)}=\|x\|_{L_p(\mathbb T^d_\theta)},\quad x\in L_p(\mathbb T^d_\theta).
\end{equation}
%and $(\pi_t)^{-1}=\pi_{t^{-1}}$.

We recall some terminology from the theory of group representations (see \cite{BGM, FS1, FS2, FS3, FS4}). Let $G$ be a compact abelian group with normalized Haar measure $dt$, and let $X$ be a Banach space equipped with a strongly continuous representation $\{R_t\}_{t\in G}$. Let $\widehat G$ denote the dual group of continuous characters $\gamma: G \to \mathbb{T}$. For each $\gamma \in \widehat G$ we define
$$
E_\gamma x = \int\limits_G \gamma(t)\,R_{-t}(x)\,dt,\quad x\in X,
$$
where the integral is understood as a Bochner integral. 

For each $\gamma \in \widehat{G}$, we define the eigenspace of a strongly continuous representation $\{R_t\}_{t\in G}$ corresponding to the character $t \mapsto \gamma(t)$ by
\begin{equation}\label{X_g}
    X_\gamma = \{x\in X : R_t(x)=\gamma(t)\,x \ \text{for all } t\in G\}.
\end{equation}
\begin{definition}\cite{ FS2}
 Let $\{R_t\}_{t\in \mathbb T^d}$ be a strongly continuous representation of the group $\mathbb T^d$ on a Banach space $X$ such that all its eigenspaces are at most one-dimensional. Set $U = \{\gamma \in \mathbb Z^d : X_\gamma \neq \{0\}\},$
and for each $\gamma \in U$ choose a non-zero vector $x_\gamma \in X_\gamma$.
The family 
\begin{equation}\label{gts}
    \{x_\gamma\}_{\gamma\in U}
\end{equation}
is called a generalized
trigonometric system in $X$ (associated with the representation
$\{R_t\}_{t\in \mathbb T^d}$).   
\end{definition}

\begin{definition}\cite{FS1, FS2, FS3} Let $G$ be a compact abelian group and $\widehat{G}$ be its dual group with normalized Haar measure $dt$. Let $\{R_t\}_{t \in G}$ be a strongly continuous representation of $G$ on a Banach space $X$ with eigenspaces $\{X_\gamma\}_{\gamma \in \widehat{G}}.$ For any subset $A \subset \widehat{G}$, the projection along $A$ is defined as the linear continuous operator $T_A $ on $X$ such that
\begin{equation}\label{TA}
T_A|_{X_\gamma} :=
\begin{cases}
\mathrm{id}, & \gamma \in A,\\
0, & \gamma \notin A,
\end{cases} 
\end{equation}
where $\mathrm{id}$ denotes the identity operator on $X$.
\end{definition}

We also need the following convention to study boundedness of the linear operators in general symmetric spaces.
\begin{convention}Let $T$ be a linear operator on $L_2(\mathbb{T}_{\theta}^{d})$ and 
on a (quasi)-Banach symmetric spaces $\mathcal{E}(\mathbb{T}^d_{\theta})$ and let $ \mathcal{F}(\mathbb{T}^d_{\theta})$ be some (quasi)-Banach symmetric spaces on $(\mathbb{T}^d_{\theta},\tau_{\theta})$ such that  $(\mathcal{E}\cap L_2)(\mathbb{T}^d_{\theta})$  is dense in $\mathcal{E}             (\mathbb{T}^d_{\theta})$.
We say that $T:\mathcal{E}(\mathbb{T}^d_{\theta})\to \mathcal{F}(\mathbb{T}^d_{\theta}),$ if for every $V\in (\mathcal{E}\cap L_2)(\mathbb{T}^d_{\theta}),$ we have $\|T(V)\|_{\mathcal{F}(\mathbb{T}^d_{\theta})}\leq c_T\|V\|_{\mathcal{E}(\mathbb{T}^d_{\theta})}$  for some constant $c_T>0$ depending on $T$ only. In this case, the operator $T$ admits a linear continuous extension from $\mathcal{E}(\mathbb{T}^d_{\theta})$ to $\mathcal{F}(\mathbb{T}^d_{\theta})$.
\end{convention}
\begin{definition}We say that a linear map $T$ acting on $L_p,$ $1<p<\infty$ spaces, admits a weak (1,1) estimate, if it admits a  linear continuous extension from $L_1$ to $L_{1,\infty}.$
\end{definition}

\section{Trigonometric basis in non-commutative $L_p(\mathbb T^d_{\theta})$ spaces}
Let $1 \leq p \leq \infty$. For each $\gamma \in \mathbb{Z}^d$ we define the operator
\begin{equation}\label{E}
E_\gamma x 
= \int\limits_{\mathbb{T}^d} t^{\gamma}\,\pi_{t^{-1}}(x)\,dt,
\quad x \in L_p(\mathbb{T}^d_\theta),
\end{equation}
where the integral is understood in the sense of the Bochner integration in $L_p(\mathbb T^d,L_p(\mathbb{T}^d_\theta))$. 

For each $\gamma \in \mathbb{Z}^d$, we define the eigenspace of the  representation
$\{\pi_t\}_{t\in\mathbb T^d}$ corresponding to the character $t \mapsto t^\gamma$ by
\begin{equation}\label{Lp_gamma}
L_p(\mathbb{T}^d_\theta)_\gamma
= \bigl\{ x \in L_p(\mathbb{T}^d_\theta) : \pi_t(x) = t^\gamma x 
\ \text{for all } t \in \mathbb{T}^d \bigr\}.    
\end{equation}

\begin{theorem}\label{3.1}
Let $1 \leq p \leq \infty$. For every $\gamma \in \mathbb{Z}^d$ the operator $E_\gamma,$ defined by \eqref{E} is a bounded projection from $L_p(\mathbb{T}^d_\theta)$
onto the eigenspace $L_p(\mathbb{T}^d_\theta)_\gamma$ and we have
$$
E_\gamma\bigl(L_p(\mathbb{T}^d_\theta)\bigr) = L_p(\mathbb{T}^d_\theta)_\gamma.
$$
\end{theorem}
\begin{proof}
Let $1 \le p \le \infty$ and $\gamma \in \mathbb Z^d$.  
%Since $\pi_t$ is an isometry on $L_p(\mathbb{T}_\theta^d)$ for every $t \in \mathbb{T}^d$ (see \eqref{PI}), the Bochner integral
%$$
%E_\gamma x
%   = \int\limits_{\mathbb T^d} t^\gamma\, \pi_{t^{-1}}(x)\, dt
%$$
%is well defined for all $x \in L_p(\mathbb T_\theta^d)$. 
Then for any $x\in L_p(\mathbb T_\theta^d)$ we obtain 
$$
\|E_\gamma x\|_{L_p(\mathbb T_\theta^d)}
   \le \int\limits_{\mathbb T^d} |t^\gamma|\,\|\pi_{t^{-1}}(x)\|_{L_p(\mathbb T_\theta^d)}\, dt
   \overset{\eqref{PI}}{=} \|x\|_{L_p(\mathbb T_\theta^d)} \int\limits_{\mathbb T^d} |t^\gamma|\, dt .
$$
Since $|t^\gamma|=1$ for every $t\in\mathbb T^d$, the last integral equals $1$. Hence,
$$
\|E_\gamma x\|_{L_p(\mathbb{T}^d_\theta)} \leq \|x\|_{L_p(\mathbb{T}^d_\theta)}, \quad x\in L_p(\mathbb{T}^d_\theta),
$$
which shows that $E_\gamma$ is a bounded operator on $L_p(\mathbb{T}_\theta^d)$.
Since, for each $u\in \mathbb{T}^d,$ the operator $\pi_u$ is a linear bounded operator on $L_p(\mathbb{T}_\theta^d)$, 
we have
$$
\pi_u(E_\gamma x) 
= \int\limits_{\mathbb{T}^d} t^\gamma \, \pi_u \pi_{t^{-1}}(x) \, dt
= \int\limits_{\mathbb{T}^d} t^\gamma \, \pi_{u t^{-1}}(x) \, dt, \, \, x\in L_p(\mathbb{T}^d_\theta).
$$
Substituting $t = us$, we obtain
$$
\pi_u(E_\gamma x) 
= u^{\gamma}\int\limits_{\mathbb{T}^d} s^{\gamma}\,\pi_{s^{-1}} (x) \, ds
= u^{\gamma}\,E_\gamma x, \quad x\in L_p(\mathbb{T}_{\theta}^d),
$$
which shows that $E_\gamma x \in L_p(\mathbb{T}^d_\theta)_\gamma$.
If $y \in L_p(\mathbb{T}_\theta^d)_\gamma$, $\gamma\in \mathbb{Z}^d,$ then $\pi_{t^{-1}}(y) = t^{-\gamma} y$ for all $t\in\mathbb T^d$, and
$$
E_\gamma y 
= \int\limits_{\mathbb{T}^d} t^\gamma \, \pi_{t^{-1}}(y) \, dt
= \int\limits_{\mathbb{T}^d} t^\gamma t^{-\gamma} y \, dt
= y \int\limits_{\mathbb{T}^d} dt = y,\,\ y\in L_p(\mathbb{T}^d_\theta).
$$
Thus, $E_\gamma$ acts as the identity on $L_p(\mathbb{T}_\theta^d)_\gamma$, and therefore, $E_\gamma^2 = E_\gamma$ on $L_p(\mathbb{T}^d_{\theta})$. Hence, $E_\gamma$ is a bounded projection from $L_p(\mathbb{T}_\theta^d)$ onto $L_p(\mathbb{T}_\theta^d)_\gamma$. This completes the proof.
\end{proof}
For each $x \in L_p(\mathbb T^d_{\theta})$, $1\leq p\leq \infty,$ and $\gamma \in \mathbb Z^d$, we set
\begin{equation}\label{Ex}
 x_\gamma := E_\gamma x,
\end{equation}
where $E_\gamma$ is the operator defined by \eqref{E}. 
By Theorem \ref{3.1}, we have that $x_\gamma \in L_p(\mathbb T^d_{\theta})_\gamma.$ 
By formula (1.4) in  \cite{BGM}, the closed linear span of the eigenspaces satisfies
\begin{equation}\label{LP}
\big[\,L_p(\mathbb{T}^d_\theta)_\gamma\,\big]_{\gamma\in\mathbb{Z}^d}
= L_p(\mathbb{T}^d_\theta).   
\end{equation}
\begin{lemma} 
Let $1 \le p \le \infty$. For each $\gamma \in \mathbb{Z}^d$, we have
$$
L_p(\mathbb{T}^d_\theta)_\gamma = [\,e^\theta_\gamma\,].
$$
\end{lemma}

\begin{proof}
 Let $x \in \mathcal{P}_\theta$ be a trigonometric polynomial of the form $x = \sum\limits_{\eta\in F} c_\eta e_\eta^\theta$ for some finite set $F \subset \mathbb Z^d$ and  $c_\eta\in\mathbb C$. Then for all $t \in \mathbb{T}^d$, we have
$$
\pi_{t^{-1}}(x) = \sum\limits_{\eta \in F} t^{-\eta} c_\eta e_\eta^\theta.
$$
Hence,
$$
E_\gamma x = \int\limits_{\mathbb{T}^d} t^\gamma \pi_{t^{-1}}(x) \, dt
= \sum\limits_{\eta \in F} c_\eta e_\eta^\theta \int\limits_{\mathbb{T}^d} t^{\gamma - \eta} \, dt.
$$
Using the formula
$$
\int\limits_{\mathbb{T}^d} t^\alpha \, dt =
\begin{cases}
1, & \alpha = 0, \\
0, & \alpha \neq 0,
\end{cases}
$$
for any $x \in \mathcal{P}_\theta,$ we obtain 
\begin{equation}\label{Eg-poly}
E_\gamma x =
\begin{cases}
c_\gamma e_\gamma^\theta, & \gamma \in F, \\
0, & \gamma \notin F.
\end{cases}
\end{equation}

Now, let $x \in L_p(\mathbb{T}^d_\theta)$. Since $\mathcal{P}_\theta$ is dense in $L_p(\mathbb{T}^d_\theta)$ (see \cite{XXY}), there exists a sequence $\{x_n\}_{n\ge 1} \subset \mathcal{P}_\theta$ such that $x_n \to x$ in $L_p(\mathbb{T}^d_\theta)$ as $n\to \infty$. By Theorem~\ref{3.1}, the operator $E_\gamma$ is bounded, hence,
$E_\gamma x_n \to E_\gamma x$ in $L_p(\mathbb{T}^d_\theta)$ as $n \to \infty$. On the other hand, by formula \eqref{Eg-poly}, $E_\gamma x_n = c_\gamma(x_n) e_\gamma^\theta \in [e_\gamma^\theta]$. Since $[e_\gamma^\theta]$ is closed, it follows that $E_\gamma x \in [e_\gamma^\theta]$. Therefore, for any $x \in L_p(\mathbb{T}^d_\theta)$, there exists $c_\gamma(x) \in \mathbb{C}$ such that
$$
E_\gamma x = c_\gamma(x) e_\gamma^\theta.
$$
Consequently, we have
$$
L_p(\mathbb{T}^d_\theta)_\gamma = [\,e_\gamma^\theta\,], \quad \gamma \in \mathbb{Z}^d.
$$
This completes the proof.
\end{proof}

For any subset $A \subset \mathbb{Z}^d$, we define the subspace $L_p(\mathbb{T}^d_\theta)_A$ of $L_p(\mathbb{T}^d_\theta)$ by
\begin{equation}\label{spanLpA}
L_p(\mathbb{T}^d_\theta)_A :=\big[L_p(\mathbb{T}^d_\theta)_\gamma \big]_{\gamma \in A} .  
\end{equation}
For $1\leq p <\infty,$ define
\begin{equation}\label{Ppp}
\mathcal P_p
:= \Bigl\{x\in L_p(\mathbb T^d_\theta): 
x=\sum\limits_{\gamma\in F} x_\gamma,\,\ 
F\subset\mathbb Z^d \text{ finite},\
x_\gamma\in L_p(\mathbb T^d_\theta)_\gamma
\Bigr\}.
\end{equation}

Although $\mathcal P_p$ and $\mathcal P_\theta$ coincide as sets of trigonometric polynomials, we introduce $\mathcal P_p$ to emphasize the decomposition
$$
x=\sum_{\gamma\in F} x_\gamma,\qquad x_\gamma\in L_p(\mathbb T^d_\theta)_\gamma.
$$
This distinction is merely notational and does not play an essential role in what follows, but it is convenient for bookkeeping.

The following result is well known in the literature (see \cite{CXY, LMSZ, McDSX}, \cite[Proposition 2.7 (ii), p. 21-22]{XXY}); however, for the reader’s convenience and since it will be used repeatedly in the sequel, we provide a proof formulated in our notation and language.

\begin{lemma}\label{Pp-dense}Let $1\leq p< \infty.$ Then
the space $\mathcal P_p$ defined by \eqref{Ppp} is dense in $L_p(\mathbb T^d_\theta)$.
\end{lemma}

\begin{proof}
By \eqref{Ppp}, the space $\mathcal P_p$ coincides with the linear
span of the eigenspaces $L_p(\mathbb T^d_\theta)_\gamma$, $\gamma\in\mathbb Z^d$.
Hence, by~\eqref{LP}, we obtain
$$
\overline{\mathcal P_p}^{\|\cdot\|_{L_p(\mathbb T^d_\theta)}}
= \big[\,L_p(\mathbb T^d_\theta)_\gamma\,\big]_{\gamma\in\mathbb Z^d}
= L_p(\mathbb T^d_\theta),
$$
thereby completing the proof.
\end{proof}

The next proposition establishes a key result on the existence of a bounded projection
associated with a total order on $\mathbb{Z}^d$, adapting the proof from \cite[Theorem 4.1]{BGM} (see also \cite[Theorem 2]{FS4}) to our context.

\begin{proposition}\label{lem1}
Let $1<p<\infty.$ Let $\leq_d$ be a total order on the group $\mathbb Z^{d}$. Then there exists a bounded projection $T_{\mathbb Z^{d}_{+}} : L_{p}(\mathbb T^{d}_{\theta})\to L_{p}(\mathbb T^{d}_{\theta})$ defined by \eqref{TA} along the subset $\mathbb Z^{d}_{+}$ of $\mathbb Z^{d}.$
\end{proposition}
\begin{proof}
Let $\leq_{d}$ be a total order on $\mathbb{Z}^d$ compatible with the additive group structure.
Fix $\eta=(\eta_1,\dots,\eta_d)\in\mathbb Z^d$ and define a homomorphism $\vartheta:\mathbb{T}\to\mathbb{T}^d$ by
$$
\vartheta(t)=(t^{\eta_1},\dots,t^{\eta_d}),
$$
and set
$$
R_t:=\pi_{\vartheta(t)}, \quad t\in\mathbb T .
$$
Since $\vartheta:\mathbb T\to\mathbb T^d$ is a group homomorphism and
$\{\pi_z\}_{z\in\mathbb T^d}$ is a representation, it follows that
$\{R_t\}_{t\in\mathbb T}$ is a representation of $\mathbb T$ on
$L_p(\mathbb T^d_\theta)$. Moreover, by \eqref{PI} we have
\begin{equation}\label{Rv}
    \|R_t(x)\|_{L_p(\mathbb T^d_\theta)}
=\|\pi_{\vartheta(t)}(x)\|_{L_p(\mathbb T^d_\theta)}
=\|x\|_{L_p(\mathbb T^d_\theta)},
\quad x\in L_p(\mathbb T^d_\theta),\ t\in\mathbb T.
\end{equation}
Let $\{F_N\}_{N\geq1}$ be a sequence of Fejér kernels in $L_1(\mathbb{T})$, 
that is, a sequence of functions in $L_1(\mathbb{T})$ such that 
$\|F_N\|_{L_1(\mathbb{T})} = 1$ and $\operatorname{supp}(\widehat{F}_N)$ is finite 
for each $N\in\mathbb{N}$, where $\widehat{F}_N$ is the Fourier coefficient of $F_N.$  Moreover,
$
\lim\limits_{N\to\infty} \widehat{F}_N(z) = 1
$ for every $z \in \mathbb{Z}$
(see \cite[Chapter III, p. 167]{G2008}). Let 
$$
h_N(t)=\sum\limits_{j\ge 0}\widehat{F}_N(j)\,t^{j}, \quad t\in\mathbb{T}.
$$
Define the transferred operator $T_N$ acting on $L_p(\mathbb{T}^d_{\theta}),$ $1<p<\infty,$ by setting
\begin{equation}\label{TN}
  T_N x=\int\limits_{\mathbb{T}} h_N(t)\,R_{t^{-1}}(x)\,dt, \quad x\in L_p(\mathbb{T}^d_{\theta}).  
\end{equation}
 Fix $s\in\mathbb T$. Since $R_{s^{-1}}R_s=I$, we may write
$$
\|T_Nx\|_{L_p(\mathbb T^d_\theta)}
=\|R_{s^{-1}}R_sT_Nx\|_{L_p(\mathbb T^d_\theta)},\quad x\in L_p(\mathbb{T}^d_{\theta}).  
$$
By \eqref{Rv}, the operators $R_s$ and $R_{s^{-1}}$ act isometrically on $L_p(\mathbb T^d_\theta)$, hence,
$$
\|R_{s^{-1}}R_sT_Nx\|_{L_p(\mathbb T^d_\theta)}
=\|R_sT_Nx\|_{L_p(\mathbb T^d_\theta)},\quad x\in L_p(\mathbb{T}^d_{\theta}).  
$$
Moreover, since $R_s$ is bounded on $L_p(\mathbb T^d_\theta)$, we may pass it under the Bochner integral
(see \cite[Theorem~1.2.4]{HNVW}):
$$
R_sT_Nx
=R_s\!\left(\int\limits_{\mathbb T} h_N(t)\,R_{t^{-1}}(x)\,dt\right)
=\int\limits_{\mathbb T} h_N(t)\,R_sR_{t^{-1}}(x)\,dt,\quad x\in L_p(\mathbb{T}^d_{\theta}).  
$$
Using that $\{R_t\}_{t\in\mathbb T}$ is a representation of $\mathbb T$,
we have
$R_sR_{t^{-1}}=R_{t^{-1}}R_s$, and therefore,
$$
R_sT_Nx=\int\limits_{\mathbb T} h_N(t)\,R_{t^{-1}}(R_sx)\,dt=T_NR_sx,\quad x\in L_p(\mathbb{T}^d_{\theta}),\,\ s\in\mathbb{T},\,\ N\ge 1.  
$$
Consequently,
$$
\|T_Nx\|_{L_p(\mathbb T^d_\theta)}
=\|R_sT_Nx\|_{L_p(\mathbb T^d_\theta)}
=\|T_NR_sx\|_{L_p(\mathbb T^d_\theta)},\quad x\in L_p(\mathbb{T}^d_{\theta}), \,\ s\in\mathbb{T}, \,\ N\ge 1.  
$$
Squaring both sides and then integrating with respect to $s\in\mathbb T$, we obtain
\begin{align}\label{T_N}
\| T_N x \|^2_{L_p(\mathbb T^d_\theta)}=\int\limits_{\mathbb{T}}\| T_N x \|^2_{L_p(\mathbb T^d_\theta)}ds
&= \int\limits_{\mathbb{T}}\|T_N R_s x\|^2_{L_p(\mathbb T^d_\theta)}ds =\int\limits_{\mathbb{T}} \left\| \int\limits_{\mathbb{T}} h_N(t) R_{st^{-1}} (x)  \,dt \right\|^2_{L_p(\mathbb T^d_\theta)} ds \nonumber\\
&\overset{\eqref{*}}{=} \int\limits_\mathbb{T} \| (h_N(t) * R_t(x))(s) \|^2_{L_p(\mathbb T^d_\theta)} ds 
= \|h_N(t)*R_t(x)\|^2_{L_2(\mathbb{T},L_p(\mathbb T^d_\theta))}.
\end{align}
By the definition of $h_N$, we have
$$
\widehat{h}_N(k)
 = \int\limits_{\mathbb T} h_N(t)\,t^{-k}\,dt
 = \int\limits_{\mathbb T} \Bigl(\sum\limits_{j\ge 0}\widehat{F}_N(j)t^j\Bigr)t^{-k}\,dt
 = \sum\limits_{j\ge 0}\widehat{F}_N(j)\int\limits_{\mathbb T} t^{j-k}\,dt, \,\ k\in\mathbb Z.
$$
If $j\neq k,$ then $\int\limits_{\mathbb T} t^{j-k}\,dt = 0$. Therefore,
\begin{equation}\label{hhat}
 \widehat{h}_N(k)
=
\begin{cases}
\widehat{F}_N(k), & k\ge 0,\\
0, & k<0.
\end{cases}   
\end{equation}
Recall that the Hilbert transform $H_{L_p(\mathbb T^d_\theta)}:L_2\bigl(\mathbb T, L_p(\mathbb T^d_\theta)\bigr)\to L_2\bigl(\mathbb T, L_p(\mathbb T^d_\theta)\bigr)$ defined by \eqref{H_X} satisfies the multiplier identity
(see \eqref{HT})
$$
\widehat{H_{L_p(\mathbb T^d_\theta)}f}(k)
= -\, i\, \mathrm{sgn}(k)\,\widehat f(k),
\quad k\in\mathbb Z,
$$
for every $f\in L_2\bigl(\mathbb T, L_p(\mathbb T^d_\theta)\bigr)$. We define $Q$ by the formula $Qf=\widehat{f}(0)$, where $\widehat{f}(0)$ is the zero-th Fourier coefficient of $f\in L_2(\mathbb{T}, L_p(\mathbb{T}^d_{\theta})).$ Then, applying the operator  $\frac12\bigl(I+Q+iH_{L_p(\mathbb T^d_\theta)}\bigr)$
to $f$ yields
\begin{equation}\label{iH}
\wideparen{\tfrac12\bigl(I+Q+iH_{L_p(\mathbb T^d_\theta)}\bigr)f}(k)
=
\begin{cases}
\widehat f(k), & k\ge 0,\\
0, & k<0,
\end{cases}
\quad
f\in L_2\bigl(\mathbb T, L_p(\mathbb T^d_\theta)\bigr).
\end{equation}
Then, for every $f\in L_2(\mathbb{T},L_p(\mathbb T^d_\theta)),$ we have
$$
\widehat{h_N*f}(k)
   = \widehat{h_N}(k)\widehat f(k)
   = \widehat{F}_N(k)\,
     \widehat{2^{-1}(I+Q+iH_{L_p(\mathbb T^d_\theta)})f}(k),
   \quad k\in\mathbb{Z}.
$$
By uniqueness of Fourier coefficients in $L_2(\mathbb{T},L_p(\mathbb T^d_\theta)),$ it follows that
$$
h_N*f
   = F_N*\bigl(2^{-1}(I+Q+iH_{L_p(\mathbb T^d_\theta)})f\bigr),
   \quad f\in L_2(\mathbb{T},L_p(\mathbb T^d_\theta)).
$$
Since $\|F_N\|_{L_1(\mathbb T)}=1$ for all $N\in\mathbb{N},$
%\cite[p.167]{G2008}, 
 formula \eqref{Young} yields
$$
\|h_N * f\|_{L_2(\mathbb T,L_p(\mathbb T^d_\theta))}
 \le \|F_N\|_{L_1(\mathbb T)}
     \bigl\|2^{-1}(I+Q+iH_{L_p(\mathbb T^d_\theta)})f\bigr\|_{L_2(\mathbb T,L_p(\mathbb T^d_\theta))}, \,\ f\in L_2(\mathbb{T},L_p(\mathbb T^d_\theta)).
$$
Note that $Qf=\widehat{f}(0)=\int\limits_{\mathbb T} f(t)\,dt$. By the normalization of the Haar measure, we obtain
$$
\|Qf\|_{L_2(\mathbb T;L_p(\mathbb T^d_\theta))}=\Bigl\|\int\limits_{\mathbb T}f(t)\,dt\Bigr\|_{L_p(\mathbb T^d_\theta)}
\le \int\limits_{\mathbb T}\|f(t)\|_{L_p(\mathbb T^d_\theta)}\,dt \le \|f\|_{L_2(\mathbb T;L_p(\mathbb T^d_\theta))},
$$
so $\|Q\|_{L_2(\mathbb T;L_p(\mathbb T^d_\theta))\to L_2(\mathbb T;L_p(\mathbb T^d_\theta))}\le 1$.

By Theorem~\ref{HUMD}, the operator $H_{L_p(\mathbb T^d_\theta)}$ is bounded on
$L_2(\mathbb T;L_p(\mathbb T^d_\theta))$. Hence, the operator
$2^{-1}\bigl(I+Q+iH_{L_p(\mathbb T^d_\theta)}\bigr)$ is bounded on the same space.
Set
$$
K_p:=
\Bigl\|2^{-1}\bigl(I+Q+iH_{L_p(\mathbb T^d_\theta)}\bigr)\Bigr\|_{
L_2(\mathbb T;L_p(\mathbb T^d_\theta))\to L_2(\mathbb T;L_p(\mathbb T^d_\theta))
}.
$$
Then for all $N\in\mathbb N$ and $f\in L_2(\mathbb T;L_p(\mathbb T^d_\theta)),$ we have \begin{equation}\label{h_N} \|h_N * f\|_{L_2(\mathbb T;L_p(\mathbb T^d_\theta))} \le K_p\,\|f\|_{L_2(\mathbb T;L_p(\mathbb T^d_\theta))}. \end{equation}
Recall that by \eqref{T_N} we have
$$
\|T_N x\|_{L_p(\mathbb T^d_\theta)}
 = \|h_N(t) * R_t(x)\|_{L_2(\mathbb T,L_p(\mathbb T^d_\theta))},
$$
 where $R_t(x)\in {L_2(\mathbb T,L_p(\mathbb T^d_\theta))}$ by \eqref{Rv}. 
Applying formula \eqref{h_N} for $f(t)=R_t(x)$, we obtain
\begin{equation}\label{T_N1}
   \|T_N x\|_{L_p(\mathbb T^d_\theta)}
   \leq \,K_p\|R_t(x)\|_{L_2(\mathbb{T},L_p(\mathbb T^d_\theta))} \overset{\eqref{Rv}}{=} K_p\|x\|_{L_p(\mathbb T^d_\theta)}, \quad  x\in L_p(\mathbb T^d_\theta), 
\end{equation}
for all $N\in\mathbb{N}.$
Now, for $z\in\mathbb{Z}$ let
$$
E_z L_p(\mathbb{T}^d_{\theta})
=\{y\in L_p(\mathbb{T}^d_{\theta}) : R_t(y)=t^{z}y \ \text{for all } t\in\mathbb{T}\}
$$
be the eigenspace of the representation $\{R_t\}_{t\in\mathbb T}$ defined by
\eqref{X_g}.
 Let $y_z\in E_z L_p(\mathbb{T}^d_{\theta})$.
Using the definition of $T_N$ (see \eqref{TN}), we obtain
\begin{align*}
T_N y_z
&=\int\limits_{\mathbb T} h_N(t)\,R_{t^{-1}}(y_z)\,dt
 =\int\limits_{\mathbb T} h_N(t)\,t^{-z}y_z\,dt \\
&=\Bigl(\int\limits_{\mathbb T} h_N(t)\,t^{-z}\,dt\Bigr)\,y_z
 =\widehat h_N(z)\,y_z, \,\ y_z\in E_z L_p(\mathbb{T}^d_{\theta}).
\end{align*}
 Since
$\widehat F_N(z)\to 1$ as $N\to\infty$ for each fixed $z\in\mathbb Z$ and by \eqref{hhat}, we conclude that
\begin{equation}\label{T_Ny}
T_N y_z=\widehat h_N(z)\,y_z \rightarrow
\begin{cases}
y_z, & z\ge 0,\\
0, & z<0,
\end{cases}
\quad \text{as } N\to\infty,
\end{equation}
where the convergence is understood in $L_p(\mathbb{T}^d_{\theta}).$ Let $x\in\mathcal P_p$ and write
$$
x=\sum\limits_{\gamma\in F}x_\gamma,
$$
where $F\subset\mathbb Z^d$ is finite and $x_\gamma\in L_p(\mathbb T^d_\theta)_\gamma$ for each $\gamma\in F$.
By Lemma~2.5 in~\cite{1} (see also ~\cite[ p.~283 ]{1}), applied to the finite set $F$,
there exists $a\in\mathbb R^d$ such that
\begin{equation}\label{3}
a\cdot\gamma>0 \ \text{for every}\ \gamma\in F \ \text{with}\ 0<_d\gamma,
\quad
a\cdot\gamma<0 \ \text{for every}\ \gamma\in F \ \text{with}\ \gamma<_d0,
\end{equation}
where $\cdot$ denotes the Euclidean inner product on $\mathbb R^d$.
We claim that $a$ may be chosen with rational coordinates while keeping \eqref{3}
valid on the fixed finite set $F$. Indeed, for every $\gamma\in F\setminus\{0\}$ we
have $a\cdot\gamma\neq 0$. Set 
$$
\delta:=\min_{\gamma\in F\setminus\{0\}}|a\cdot\gamma|>0,
\quad
R:=\max_{\gamma\in F}\sum\limits_{j=1}^d |\gamma_j|<\infty.
$$
Choose $\varepsilon>0$ such that $0<\varepsilon<\delta/(2R)$. Since rational numbers
are dense in $\mathbb R$, we may select rational numbers $a'_1,\dots,a'_d$ such that
$$
|a_j-a'_j|<\varepsilon,\qquad j=1,\dots,d,
$$
and define $a':=(a'_1,\dots,a'_d)$. Then for every $\gamma\in F$,
$$
|(a-a')\cdot\gamma|
=\Bigl|\sum\limits_{j=1}^d (a_j-a'_j)\gamma_j\Bigr|
\le \sum\limits_{j=1}^d |a_j-a'_j|\,|\gamma_j|
\le \varepsilon\sum\limits_{j=1}^d |\gamma_j|
\le \varepsilon R
<\frac{\delta}{2}.
$$
If $a\cdot\gamma>0$, then $a\cdot\gamma\ge\delta$ and
$$
a'\cdot\gamma
= a\cdot\gamma + (a'-a)\cdot\gamma
\ge a\cdot\gamma - |(a'-a)\cdot\gamma|
> \delta-\frac{\delta}{2}
=\frac{\delta}{2}>0.
$$
If $a\cdot\gamma<0$, then $a\cdot\gamma\le -\delta$ and similarly
$$
a'\cdot\gamma
\le a\cdot\gamma + |(a'-a)\cdot\gamma|
< -\delta+\frac{\delta}{2}
=-\frac{\delta}{2}<0.
$$
Consequently, for every $\gamma\in F\setminus\{0\}$ the numbers $a\cdot\gamma$ and
$a'\cdot\gamma$ have the same sign. In particular, \eqref{3} remains valid on $F$
after replacing $a$ by $a'$.
Multiplying $a'$ by a suitable positive integer, we may further assume that $a'$
has integer coordinates. This scaling does not change the sign relations in
\eqref{3}. At this point we choose $\eta=a$ in the above construction.
Thus $a=(\eta_1,\dots,\eta_d)\in\mathbb Z^d$, $\vartheta(t)=(t^{\eta_1},\dots,t^{\eta_d})$, and for $x_\gamma\in L_p(\mathbb T^d_\theta)_\gamma$ we have
$$
R_t(x_\gamma)=\pi_{\vartheta(t)}(x_\gamma)=(\vartheta(t))^\gamma x_\gamma=t^{a\cdot\gamma}\,x_\gamma,
\quad t\in\mathbb T, \gamma\in \mathbb{Z}^d,
$$
so $x_\gamma\in E_{a\cdot\gamma}L_p(\mathbb T^d_\theta)$. By \eqref{TN}, we obtain
$$
T_N x_\gamma = \widehat{h}_N(a\cdot\gamma)\, x_\gamma.
$$
Consequently,
$$
T_N x = \sum\limits_{\gamma \in F} \widehat{h}_N(a\cdot\gamma)\, x_\gamma, \quad x\in \mathcal{P}_p.
$$
By \eqref{T_Ny}, we have
$$
\lim_{N \to \infty} T_N x
  = \sum\limits_{\substack{\gamma \in F \\ a\cdot\gamma \ge 0}} x_\gamma
  \overset{\eqref{3}}{=} \sum\limits_{\substack{\gamma \in F \\ \gamma \ge_d 0}} x_\gamma, \quad x\in\mathcal{P}_p,
$$
where the limit is understood in $L_p(\mathbb{T}^d_{\theta})$.
Since $\mathcal{P}_p$ is dense in $L_p(\mathbb{T}^d_{\theta})$ by Lemma~\ref{Pp-dense}, there exists a strong limit of the sequence of operators $\{T_N\}_{N\ge 1}$ on $L_p(\mathbb{T}^d_{\theta}),$ which we denote by $T_{\mathbb{Z}^d_+},$ i.e.
$$
T_{\mathbb{Z}^d_+} x := \lim_{N \to \infty} T_N x,
\quad x \in L_p(\mathbb{T}^d_{\theta}).
$$
Since $T_N$ are linear and bounded for all $N\in\mathbb N$ (see \eqref{T_N1}), the operator $T_{\mathbb{Z}^d_+}$ is also linear and bounded. Moreover, by the 
Banach-Steinhaus theorem we obtain
\begin{equation}\label{4.8}
\|T_{\mathbb{Z}^d_+}\|_{L_p(\mathbb T^d_\theta)\to L_p(\mathbb T^d_\theta)}
  \le \sup_{N\geq 1} \|T_N\|_{L_p(\mathbb T^d_\theta)\to L_p(\mathbb T^d_\theta)}
  \le K_p,
\end{equation}
thereby completing the proof.
\end{proof}

For $\gamma\in\mathbb Z^d$ and for any noncommutative quasi-Banach symmetric space $\mathcal E(\mathbb T^d_\theta)$ (see Definition \ref{NC Sym}), let us define
\begin{equation}\label{A_gamma}
A_\gamma : \mathcal E(\mathbb T^d_\theta)\to \mathcal E(\mathbb T^d_\theta),\quad
A_\gamma(x)=e^{\theta}_\gamma x, \,\ x\in \mathcal E(\mathbb T^d_\theta).
\end{equation}

\begin{lemma}\label{A_gamma1}
Let $\gamma\in \mathbb{Z}^d.$ Then the operator $A_\gamma$ defined by \eqref{A_gamma} is an isometry on every noncommutative quasi-Banach symmetric space $\mathcal E(\mathbb{T}^d_\theta)$ and we have
$$
\|A_\gamma(x)\|_{\mathcal E(\mathbb T^d_\theta)} = \|x\|_{\mathcal E(\mathbb T^d_\theta)},
\quad x\in \mathcal E(\mathbb{T}^d_\theta).
$$
\end{lemma}
\begin{proof}
Fix $\gamma\in \mathbb{Z}^d.$ Then for any $x\in \mathcal E(\mathbb T^d_\theta)$ we have $$
|e^{\theta}_\gamma  x|=\big((e^{\theta}_\gamma  x)^*(e^{\theta}_\gamma  x)\big)^{1/2}=\big(x^*(e^{\theta}_\gamma )^*e^{\theta}_\gamma x\big)^{1/2}=\big(x^*x\big)^{1/2}=|x|.
$$
Therefore, we have
$$\|A_\gamma(x)\|_{\mathcal E(\mathbb T^d_\theta)}=\|e^{\theta}_\gamma x\|_{\mathcal E(\mathbb T^d_\theta)}=\||e^{\theta}_\gamma x|\|_{\mathcal E(\mathbb T^d_\theta)}=\||x|\|_{\mathcal E(\mathbb T^d_\theta)}=\|x\|_{\mathcal E(\mathbb T^d_\theta)}, \,\ x\in \mathcal E(\mathbb T^d_\theta),$$
thereby completing the proof.
\end{proof}
\begin{definition}\label{intervals}
Let $\leq_d$ be a total order on $\mathbb Z^d$. For any $\gamma_1,\gamma_2 \in \mathbb Z^d,$ we set
$$
[\gamma_1,\gamma_2]
:= 
\bigl\{\gamma \in \mathbb Z^d : \gamma_1 \leq_d \gamma \leq_d \gamma_2\bigr\},
$$
$$
[\gamma_1,\infty)
:=
\bigl\{\gamma \in \mathbb Z^d : \gamma_1 \leq_d \gamma\bigr\},
\quad
(-\infty,\gamma_2]
:=
\bigl\{\gamma \in \mathbb Z^d : \gamma \leq_d \gamma_2\bigr\}.
$$
These sets will be called closed intervals (finite or semi-infinite) in $ \mathbb{Z}^d.$ 
\end{definition}
We now state an extension of Proposition \ref{lem1} that shows the projection $T_A$ in \eqref{TA} is well-defined and bounded for any interval $A\subset \mathbb{Z}^d.$

\begin{lemma}\label{lem2} Let $1<p<\infty.$ Let $A \subset \mathbb{Z}^d$ be an interval as in Definition \ref{intervals}. Then, there exists a bounded projection $T_A$ defined by \eqref{TA} along $A$ on $L_p(\mathbb{T}^d_\theta).$ 
\end{lemma}
\begin{proof}
Let  $1<p<\infty.$
By Proposition~\ref{lem1}, there exists a bounded projection $
T_{\mathbb Z^{d}_{+}}:L_{p}(\mathbb T^{d}_{\theta})\to L_{p}(\mathbb T^{d}_{\theta})$
along $\mathbb Z^{d}_{+}=\{\gamma\in\mathbb Z^{d}: \gamma\geq_d 0\}$
such that
$$
\|T_{\mathbb Z^{d}_{+}}\|_{L_{p}(\mathbb T^{d}_{\theta})\to L_{p}(\mathbb T^{d}_{\theta})}
\le K_p.
$$
where $K_p>0$ is the same as in Proposition \ref{lem1}.
Fix $\gamma_{0}\in\mathbb Z^{d}$ and denote $[\gamma_{0},\infty)  := \{\gamma\in\mathbb Z^{d}:\ \gamma\geq_d\gamma_{0}\}.$
 To shift the projection $T_{\mathbb Z^{d}_{+}}$ to $(\gamma_{0},\infty)$, define, as in \eqref{A_gamma},
$$
A_{\gamma_{0}}(x)=e^{\theta}_{\gamma_{0}}x,
\quad x\in L_{p}(\mathbb T^{d}_{\theta}),
$$
which is an isometry on $L_p(\mathbb T^d_\theta)$ by Lemma \ref{A_gamma1} and satisfies
$$
A_{\gamma_{0}}\,L_{p}(\mathbb T^{d}_{\theta})_{\gamma}
  = L_{p}(\mathbb T^{d}_{\theta})_{\gamma+\gamma_{0}} .
$$
Indeed, if $x_\gamma\in L_{p}(\mathbb T^{d}_{\theta})_{\gamma}$, then
$$
\pi_t(A_{\gamma_{0}}x_\gamma)
 = t^{\gamma_{0}}e^{\theta}_{\gamma_{0}} t^{\gamma}x_\gamma
 = t^{\gamma+\gamma_{0}}(A_{\gamma_{0}}x_\gamma),
$$
so $A_{\gamma_{0}}x_\gamma\in L_{p}(\mathbb T^{d}_{\theta})_{\gamma+\gamma_{0}}$.
Consequently, we have
$$
(A_{\gamma_{0}}T_{\mathbb Z^{d}_{+}}A_{-\gamma_{0}})x_\gamma
 =
\begin{cases}
x_\gamma, & \gamma \geq_d\gamma_{0},\\
0,        & \gamma <_d\gamma_{0}.
\end{cases}
$$
Set
$$
T_{[\gamma_{0},\infty)}
 := A_{\gamma_{0}}T_{\mathbb Z^{d}_{+}}A_{-\gamma_{0}}.
$$
Then it is precisely the projection along $[\gamma_{0},\infty)$. Since the operators $A_{\gamma_0}$ are isometries on $L_p(\mathbb T^d_\theta)$
by Lemma \ref{A_gamma1}, we obtain
$$
\|T_{[\gamma_0,\infty)}\|_{L_p(\mathbb T^d_\theta)\to L_p(\mathbb T^d_\theta)}
= \|A_{\gamma_{0}}T_{\mathbb Z^{d}_{+}}A_{-\gamma_{0}}\|_{L_p(\mathbb T^d_\theta)\to L_p(\mathbb T^d_\theta)}=\|T_{\mathbb Z^d_+}\|_{L_p(\mathbb T^d_\theta)\to L_p(\mathbb T^d_\theta)}
\le K_p .
$$

To obtain projections onto intervals of the form $(-\infty,\gamma_{2}]$,
we apply the same construction as above, but with the opposite direction order on $\mathbb Z^{d}.$ In this way, we also obtain a bounded projection
$T_{(-\infty,\gamma_{2}]}$, and therefore there exists a bounded projection
along any semi-infinite interval in $\mathbb Z^{d}$.

Now let $[\gamma_{1},\gamma_{2}]$ be an arbitrary finite closed intervals in $\mathbb Z^{d}$.
It can be written as
$$
[\gamma_{1},\gamma_{2}]
 = [\gamma_{1},\infty)\cap(-\infty,\gamma_{2}].
$$
Let $T_{[\gamma_{1},\infty)}$ and $T_{(-\infty,\gamma_{2}]}$ denote the corresponding
semi-infinite interval projections onto $[\gamma_1, \infty)$ and $(-\infty, \gamma_2],$ respectively. For $x = \sum\limits_{\gamma\in \mathbb Z^d} x_\gamma \in \mathcal{P}_p$
with $x_\gamma \in L_p(\mathbb T^d_\theta)_\gamma$, we have 
$$
T_{[\gamma_{1},\infty)}x = \sum\limits_{\gamma\geq_d\gamma_{1}}x_\gamma,
\quad
T_{(-\infty,\gamma_{2}]}x = \sum\limits_{\gamma\leq_d\gamma_{2}}x_\gamma,
$$
and 
$$
T_{[\gamma_{1},\infty)}T_{(-\infty,\gamma_{2}]}x
 = T_{[\gamma_{1},\infty)}\Bigl(\sum\limits_{\gamma\leq_d\gamma_{2}}x_\gamma\Bigr)
 = \sum\limits_{\gamma\in[\gamma_{1},\gamma_{2}]}x_\gamma.
$$
Thus,
\begin{equation}\label{T_interval}
    T_{[\gamma_{1},\gamma_{2}]}x
 := \sum\limits_{\gamma\in[\gamma_{1},\gamma_{2}]}x_\gamma
 = T_{[\gamma_{1},\infty)}T_{(-\infty,\gamma_{2}]}x,
\end{equation}
so $T_{[\gamma_{1},\gamma_{2}]}
   = T_{[\gamma_{1},\infty)}T_{(-\infty,\gamma_{2}]}$ and
$$
\|T_{[\gamma_{1},\gamma_{2}]}x\|_{L_{p}(\mathbb T^{d}_{\theta})}
 \leq K_p\,\|T_{(-\infty,\gamma_{2}]}x\|_{L_{p}(\mathbb T^{d}_{\theta})}
 \le K^2_p\|x\|_{L_{p}(\mathbb T^{d}_{\theta})}, \, x\in \mathcal{P}_p.
$$
Since $\mathcal{P}_p$ is dense in $L_p(\mathbb T^{d}_{\theta})$, and the above estimates show that 
$T_{[\gamma_{1},\gamma_{2}]}$ is bounded on $\mathcal{P}_p$, 
it follows by continuity that $T_{[\gamma_{1},\gamma_{2}]}$ extends uniquely 
to a bounded projection on $L_p(\mathbb T^{d}_{\theta})$ with
$$
\|T_{[\gamma_{1},\gamma_{2}]}x\|_{L_{p}(\mathbb T^{d}_{\theta})}
\le K_p\|x\|_{ L_{p}(\mathbb T^{d}_{\theta})}, \quad x\in L_{p}(\mathbb T^{d}_{\theta}).
$$
This completes the proof.
\end{proof}

\begin{rem}\label{rem:union-interval-projection}
Let $\mathcal J,\mathcal K\subset \mathbb Z^{d}$ be intervals.  
For $x=\sum\limits_{\gamma\in F}x_\gamma\in \mathcal P_p$, where $F\subset\mathbb Z^d$ is a finite set and
$x_\gamma\in L_p(\mathbb T^d_\theta)_\gamma$, formula~\eqref{TA} yields
$$
T_{\mathcal J}x=\sum\limits_{\gamma\in F\cap\mathcal J}x_\gamma,\quad
T_{\mathcal K}x=\sum\limits_{\gamma\in F\cap\mathcal K}x_\gamma,\quad
T_{\mathcal J\cap\mathcal K}x=\sum\limits_{\gamma\in F\cap(\mathcal J\cap\mathcal K)}x_\gamma.
$$
Hence,
$$
T_{\mathcal J}x+T_{\mathcal K}x-T_{\mathcal J\cap\mathcal K}x
=\sum\limits_{\gamma\in F\cap(\mathcal J\cup\mathcal K)}x_\gamma
= T_{\mathcal J\cup\mathcal K}x,
\quad x\in\mathcal P_p.
$$
Since $\mathcal P_p$ is dense in $L_p(\mathbb T^d_\theta)$ (see Lemma~\ref{Pp-dense}) and
$T_{\mathcal J},T_{\mathcal K},T_{\mathcal J\cap\mathcal K},$ and $T_{\mathcal J\cup\mathcal K}$
are bounded on $L_p(\mathbb T^d_\theta)$, the above identity extends to 
$L_p(\mathbb T^d_\theta)$. In particular, on $L_p(\mathbb T^d_\theta)$ we have
\begin{equation}\label{TJK}
  T_{\mathcal J\cup\mathcal K}
= T_{\mathcal J}+T_{\mathcal K}-T_{\mathcal J\cap\mathcal K}.  
\end{equation}
\end{rem}

%In particular, for any subset $A \subset \mathbb{Z}^d$ we recall that $$T_A(x) = \sum\limits_{\gamma \in A} x_\gamma,$$
%where $x_\gamma \in L_p(\mathbb{T}^d_{\theta})_{\gamma}$. By the above theorem, $T_A$ is a bounded projection on $L_p(\mathbb{T}^d_\theta)$. Moreover, the range of $T_A$ is $L_p(\mathbb{T}^d_\theta)_A$, while the kernel of $T_A$ is
%L_p(\mathbb{T}^d_\theta)_{\mathbb{Z}^d \setminus A}.$

\begin{definition}
For each fixed $z \in \mathbb{Z}$ and $k = 1,2,\dots,d$, we define the upper and 
lower (open, closed, and boundary) half-spaces in $\mathbb{Z}^d$ by
\begin{equation}\label{Hk}
\begin{aligned}
H_k(z)&= \{\gamma \in \mathbb{Z}^d : \gamma_k > z\}, \quad
\overline{H}_k(z)= \{\gamma \in \mathbb{Z}^d : \gamma_k \ge z\}, \\
\partial H_k(z)&= \{\gamma \in \mathbb{Z}^d : \gamma_k = z\}, \quad
H_k^{-}(z)= \{\gamma \in \mathbb{Z}^d : \gamma_k < z\}, \\
\overline{H}_k^{-}(z)&= \{\gamma \in \mathbb{Z}^d : \gamma_k \le z\}. 
\end{aligned}
\end{equation}
\end{definition}

%This order is total and translation-invariant:$$\gamma<_{\mathrm{lex}}\eta \ \Rightarrow\ \gamma+\xi<_{\mathrm{lex}}\eta+\xi.$$
In the next lemma, we show that for every half-space $A$ in \eqref{Hk} the projection $T_A$ defined by \eqref{TA} is a bounded operator on $L_p(\mathbb T^d_\theta)$.

\begin{lemma}\label{TAH} Let $1<p<\infty.$
Let $A$ be one of the half-spaces defined in \eqref{Hk}. Then the  projection $T_A$ along $A$ defined by \eqref{TA} is bounded on $L_p(\mathbb T^d_\theta)$ and we have
$$
\|T_A x\|_{L_p(\mathbb{T}^d_{\theta})} \leq C_p\|x\|_{L_p(\mathbb{T}^d_{\theta})}  ,\quad x\in L_p(\mathbb T^d_\theta),
$$
where $C_p > 0$ is a constant independent of $x$.
\end{lemma}

\begin{proof} Let $1<p<\infty.$ For $\gamma=(\gamma_1,\dots,\gamma_d)\in\mathbb Z^d,$ we define the
coordinate functionals
\begin{equation}\label{fk}
f_k(\gamma)=\gamma_k ,\quad k=1,\dots,d.
\end{equation}
Fix $z\in\mathbb{Z}.$ Thus we may choose $\gamma^0_k\in H_k(z)=\{\gamma\in\mathbb Z^d:\gamma_k>z\}$ such that
$$\gamma^0_k:=(0,\dots, z+1, \dots0),$$
that is, the vector whose $k$-th coordinate is nonzero, which equals $z+1.$ 
We represent $H_k(z)$ as a union of two semi-infinite intervals for suitable lexicographic orders. We use lexicographic comparisons starting with the $k$-th coordinate.
Let $\le_{\mathrm{lex}}^{(1)}$ be the lexicographic order induced by the ordered family of coordinate functionals defined in \eqref{fk},
$$
(f_k,f_1,\dots,f_{k-1},f_{k+1},\dots,f_d),
$$
that is, for $\gamma,\eta\in\mathbb Z^d$ we write $\gamma<_{\mathrm{lex}}^{(1)}\eta$
if, letting
$$
(g_1,\dots,g_d):=(f_k,f_1,\dots,f_{k-1},f_{k+1},\dots,f_d),
$$
so that
$$
g_1=f_k,\,g_2=f_1, \dots, g_k=f_{k-1},\, g_{k+1}=f_{k+1},\dots, g_d=f_d,
$$
there exists $j\in\{1,\dots,d\}$ such that
$$
g_1(\gamma)=g_1(\eta),\ \dots,\ g_{j-1}(\gamma)=g_{j-1}(\eta),
\quad\text{and}\quad
g_j(\gamma)<g_j(\eta).
$$
Similarly, let $\le_{\mathrm{lex}}^{(2)}$ be the lexicographic order associated with
$$
(f_k,-f_1,\dots,-f_{k-1},-f_{k+1},\dots,-f_d).
$$
Then
\begin{equation}\label{eq:Hk-union-intervals}
H_k(z)
=
\{\gamma\in\mathbb Z^d:\ \gamma_k^0 \le_{\mathrm{lex}}^{(1)} \gamma\}
\ \cup\
\{\gamma\in\mathbb Z^d:\ \gamma_k^0 \le_{\mathrm{lex}}^{(2)} \gamma\}.
\end{equation}
Indeed, if $\gamma\in H_k(z)$ then $\gamma_k\ge z+1$, hence $f_k(\gamma)\ge f_k(\gamma_k^0)$.
If $f_k(\gamma)>f_k(\gamma_k^0)$, we have $\gamma_k^0<_{\mathrm{lex}}^{(i)}\gamma$ for both $i=1,2$.
If $f_k(\gamma)=f_k(\gamma_k^0)$, look at the first index $j\neq k$ (in the order $1,2,\dots,k-1,k+1,\dots,d$)
for which $\gamma_j\neq0$; if $\gamma_j>0$ then $\gamma_k^0\le_{\mathrm{lex}}^{(1)}\gamma$, while if $\gamma_j<0$
the sign flip yields $\gamma_k^0\le_{\mathrm{lex}}^{(2)}\gamma$.
Conversely, if $\gamma_k^0\le_{\mathrm{lex}}^{(i)}\gamma$ for some $i\in\{1,2\}$, then the first comparison is made
with $f_k$, hence $f_k(\gamma)\ge f_k(\gamma_k^0)$, i.e.\ $\gamma_k\ge z+1$, so $\gamma\in H_k(z)$.

By the discussion above, both
$\le_{\mathrm{lex}}^{(1)}$ and $\le_{\mathrm{lex}}^{(2)}$ are total orders on
$\mathbb Z^d$, hence, they are admissible in the sense of
Lemma~\ref{lem2}.

This observation explains the use of Lemma~\ref{lem2} in the half-space decomposition. \ $H_k(z)$ is the union of two semi-infinite intervals, each taken with
respect to an admissible order from Lemma~\ref{lem2}. Therefore, Lemma~\ref{lem2} yields
boundedness of the corresponding interval projections for each order, and the
desired bound for $T_{H_k(z)}$ follows by combining the two estimates (via
\eqref{TJK}). Hence, there exists a constant $C_p>0$ such that
$$
\|T_{H_k(z)}x\|_{L_p(\mathbb{T}^d_{\theta})} \leq C_p\|x\|_{ L_p(\mathbb{T}^d_{\theta})}, \quad x\in L_p(\mathbb{T}^d_{\theta}).
$$
The argument above applies verbatim to each of the half spaces $A$ in \eqref{Hk},  so we obtain
$$
\|T_A x\|_{L_p(\mathbb{T}^d_{\theta})} 
   \le C_p \|x\|_{L_p(\mathbb{T}^d_{\theta})},
\quad x\in L_p(\mathbb{T}^d_{\theta}).
$$
 This completes the proof.
\end{proof}
\begin{definition}\label{numeration}
We say that a subset $A \subset \mathbb{Z}^d$ is of type $(k,l)$ if there exist half-spaces $A_{ij} \subset \mathbb{Z}^d$, $i=1,\dots,k$, $j=1,\dots,l$, defined as in~\eqref{Hk}, such that
$$
A = \bigcup_{i=1}^k \ \bigcap_{j=1}^l A_{ij}.
$$
  A bijection $\varphi : \mathbb{Z}^d \to \mathbb{N} $ is called an enumeration of the elements of $ \mathbb{Z}^d $. We call an enumeration the $(k,l)$-enumeration if for every $n \in \mathbb{N}$ the set $\varphi^{-1}([1,n])$ is of type $(k,l)$.  
\end{definition}
For $m\in\mathbb{N}$ we define the $d$-dimensional cube
$$
K_m=
\bigl\{
x=(x_1,\dots,x_d)\in\mathbb{Z}^d:\;
-m\le x_i\le m\ (i=1,\dots,d)
\bigr\}.
$$
The cube $K_m$ can be written as the intersection of $2d$ half-spaces
$$
K_m=\bigcap_{i=1}^{d}\overline{H}_{i}(-m)
\;\cap\;
\bigcap_{i=1}^{d}\overline{H}^-_{i}(m),
$$
Indeed, $x=(x_1,\dots,x_d)\in K_m$ means that $-m\le x_i\le m$ for every
$i=1,\dots,d$. The condition $x_i\ge -m$ is equivalent to $x\in\overline H_i(-m)$,
while $x_i\le m$ is equivalent to $x\in\overline H_i^{-}(m)$. Intersecting these
$2d$ constraints over $i=1,\dots,d$ gives exactly the cube $K_m$. Thus, $K_m$ is a set of type $(1,2d)$. Each cube $K_m$ has $2d$ faces given by
$$
\Gamma_{m}^{2\ell-1}
=\bigl\{x\in K_m:\; x_\ell=-m,\, |x_j|<m,\ j\neq\ell \bigr\},
\quad
\Gamma_{m}^{2\ell}
=\bigl\{x\in K_m:\; x_\ell=m,\, |x_j|<m,\ j\neq\ell \bigr\},
$$
for $\ell=1,\dots,d$. Thus, each face can be written as
$$
\Gamma_{m}^{2\ell}
= \partial H_\ell(m)\ \cap\
\bigcap_{j\neq \ell}\overline H_j(-m+1)\ \cap\
\bigcap_{j\neq \ell}\overline H_j^{-}(m-1),
\quad \ell=1,\dots,d,
$$
and
$$
\Gamma_{m}^{2\ell-1}
= \partial H_\ell(-m)\ \cap\
\bigcap_{j\neq \ell}\overline H_j(-m+1)\ \cap\
\bigcap_{j\neq \ell}\overline H_j^{-}(m-1),
\quad \ell=1,\dots,d.
$$
In particular, each face $\Gamma_m^{k}$ is the intersection of $2(d-1)=2d-2$
coordinate half-spaces together with one boundary half-space $\partial H_\ell(\pm m)$.
Therefore, each $\Gamma_m^{k}$ is of type $(1,2d-1)$.

Fix $\ell\in\{1,\dots,d\}$. For $k\in\{2\ell-1,2\ell\}$ (so that the $\ell$-th coordinate is fixed at $\pm m$ on $\Gamma_m^{k}$), we define the projection
$$
\Phi_{m}^{\ell}:\Gamma_{m}^{k}\to\mathbb{Z}^{\,d-1},\quad
\Phi_{m}^{\ell}(x_1,\dots,x_d)=(x_1,\dots,x_{\ell-1},x_{\ell+1},\dots,x_d),
$$
obtained by deleting the coordinate fixed at $\pm m$.

\begin{lemma}\label{phi}
For every $d \ge 1$ there exists an enumeration $\varphi_d : \mathbb{Z}^d \to \mathbb{N}$ such that, for every $N \in \mathbb{N}$, the segment 
$$
\varphi_d^{-1}([1,N])
$$
is a set of type $(k_d,l_d)$, where the parameters $k_d,l_d$ depend only on $d$.
\end{lemma}
\begin{proof}
We define a total order $\leq_d$ on $\mathbb{Z}^d$ and an associated enumeration $\varphi_d$ by induction on $d.$ For $d=1$ we set
$$
x\leq_1 y 
\quad\Longleftrightarrow\quad
|x|<|y|\ \text{ or }\ (y=-x,\ y>0), \quad x,y \in \mathbb{Z},
$$
and
$$
\varphi_1(x)=\#\{y\in\mathbb{Z}:\; y\leq_1 x\},
$$
where $\#$ denotes the cardinality of a set. Since each initial segment is an interval in $\mathbb{Z}$ defined by two half-spaces, we see that this enumeration is  of type $(1,2)$. 

Assume now that $\leq_{d-1}$ and $\varphi_{d-1}$ are already defined on $\mathbb{Z}^{d-1}$. For $x,y\in\mathbb{Z}^d$ we declare $x\leq_d y$ if one of the following holds:
\begin{enumerate}
\item there exists $m\in\mathbb{N}$ such that $x\in K_m$ and $y\notin K_m$;
\item there exist $m\in\mathbb{N}$ and $k<j$ such that $x\in\Gamma_m^{k}$ and $y\in\Gamma_m^{j}$;
\item there exist $m\in\mathbb{N}$ and $k$ such that $x,y\in\Gamma_m^k$ and 
  $\Phi_m^{\ell}(x)\leq_{d-1}\Phi_m^{\ell}(y)$.
\end{enumerate}
The corresponding enumeration is defined by
$$
\varphi_d(x)=\#\{y\in\mathbb{Z}^d:\; y\leq_d x\},
\quad x\in\mathbb{Z}^d.
$$
The enumeration $\varphi_d$ orders the lattice points of $\mathbb Z^d$. 
For $N\in\mathbb N$, we consider the set of all points that appear among the first
$N$ terms of this enumeration, namely
$$
\varphi_d^{-1}([1,N])=\{x\in\mathbb Z^d:\ \varphi_d(x)\le N\}.
$$
This set can be described explicitly. Suppose that the $N$-th point in the
enumeration belongs to the face $\Gamma_m^{k}$, i.e.
$$
\varphi_d^{-1}(N)\in\Gamma_m^{k}.
$$
Then, by the definition of the order $\le_d$, we have
$$
\varphi_d^{-1}([1,N])
=
K_{m-1}
\;\cup\;
\bigcup_{j=1}^{k-1}\Gamma_m^{j}
\;\cup\;
(\Phi_m^{\ell})^{-1}\!\bigl(\varphi_{d-1}^{-1}([1,M])\bigr),
$$
where 
$$
M=\varphi_{d-1}\bigl(\Phi_m^{\ell}(\varphi_d^{-1}(N))\bigr)
$$
and $\ell\in\{1,\dots,d\}$ is determined by the condition $k\in\{2\ell-1,\,2\ell\}$.

From the above decomposition, the set $\varphi_d^{-1}([1,N])$ is the union of three
parts: $K_{m-1}$ of type $(1,2d)$, $\bigcup_{j=1}^{k-1}\Gamma_m^j$ of type $(1,2d-1)$,
and $(\Phi_m^\ell)^{-1}(\varphi_{d-1}^{-1}([1,M]))$ which is of type
$(k_{d-1},l_{d-1})$ by the inductive hypothesis.
Hence the number of union blocks increases by $2$, so $k_d=k_{d-1}+2$.
To keep a common intersection length, we add the new half-space constraints coming
from $K_{m-1}$ and the faces, namely $2d$ and $2d-1$, so
$l_d=l_{d-1}+2d+(2d-1)=l_{d-1}+4d-1$.  
\end{proof}

Now, we are ready to state the main result of this section, which provides an explicit
Schauder basis in $L_p(\mathbb T_\theta^d)$ for $1<p<\infty.$
\begin{theorem}\label{main-basis-thm}
Let $1<p<\infty$. Let $ \{ e^{\theta}_\gamma \}_{\gamma \in \mathbb{Z}^d} $ be a generalized trigonometric system in $  L_p(\mathbb{T}_\theta^d) $ defined by \eqref{gts}. If $\varphi$ is a $(k, l)$-enumeration as in Definition \ref{numeration}, then the sequence
$$
\Big\{ y_m = e^{\theta}_{\varphi^{-1}(m)} \Big\}_{m=1}^\infty
$$
forms a Schauder  basis in  $L_p(\mathbb{T}_\theta^d) .$
\end{theorem}

\begin{proof}
Let $\varphi$ be the enumeration as in Lemma \ref{phi}, which is of type $(k,l)$, and define
$$
y_m := e^{\theta}_{\varphi^{-1}(m)}, \quad m\in\mathbb N.
$$
We now verify conditions \textup{(i)}–\textup{(iii)} in Proposition \ref{Schauder} for the sequence
$\{y_m\}_{m=1}^\infty$.

Condition (i) follows from the fact that $e^{\theta}_{\gamma}\neq 0$ for every 
$\gamma\in \mathbb{Z}^d.$
Condition (ii) in Proposition \ref{Schauder} follows directly from \eqref{LP}. Now let us check the condition (iii) in Proposition \ref{Schauder}. Fix $m\in\mathbb N$. Since $\varphi$ is a $(k_d,l_d)$-numbering of $\mathbb Z^d$, the set
$\varphi^{-1}([1,m])$ is of type $(k_d,l_d)$ and hence admits a formula
$$
\varphi^{-1}([1,m])=\bigcup_{i=1}^{k_d}\ \bigcap_{j=1}^{l_d} A_{ij},
$$
where each $A_{ij}$ is a half-space of the form described in \eqref{Hk}. 
In particular, for fixed $d$ the integers $k_d$ and $l_d$ control the combinatorial
complexity of this formula and do not depend on $m$.

By Lemma~\ref{TAH}, the projection $T_A$ along any half-space $A$ is bounded on
$L_p(\mathbb T_\theta^d)$. Using the identity $T_{A\cap B}=T_A T_B$ and the formula for unions given in
\eqref{TJK}, together with the bounds from Lemma~\ref{TAH}, we conclude that
$T_{\varphi^{-1}([1,m])}$ is well defined and bounded.
Consequently,
\begin{equation}\label{T-proj-bdd}
\bigl\|T_{\varphi^{-1}([1,m])}\bigr\|_{L_p(\mathbb T^d_\theta)\to L_p(\mathbb T^d_\theta)}
\le C_{p,d},
\end{equation}
where $C_{p,d}>0$ depends only on $p$ and $d$  and is independent of $m$.

%By Lemma \ref{TAH}, the projection associated with every half-space is bounded. Hence, the projection $T_{\varphi^{-1}([1,m])}$ along the set $\varphi^{-1}([1,m])$ is bounded, and we have $$\bigl\| T_{\varphi^{-1}([1,m])} \bigr\|_{L_p(\mathbb{T}_\theta^d)\to L_p(\mathbb{T}_\theta^d)}    \leq C_p,$$ where the constant $C_p$ is the same as in Lemma~\ref{TAH},  i.e., it does not depend on the number $m$. 
Then, for all $m,n\in\mathbb N$ with $m\le n$ and every sequence $(a_j)_{j=1}^n\subset\mathbb C$, we have
\begin{align*}
T_{\varphi^{-1}([1,m])}\!\left(\sum\limits_{j=1}^n a_j e^\theta_{\varphi^{-1}(j)}\right)
&=\sum\limits_{j=1}^n a_j\, T_{\varphi^{-1}([1,m])}\!\bigl(e^\theta_{\varphi^{-1}(j)}\bigr).
\end{align*}
Since $e^\theta_{\varphi^{-1}(j)}\in L_p(\mathbb{T}^d_{\theta})_{\varphi^{-1}(j)}$, it follows from \eqref{TA} that
$$
T_{\varphi^{-1}([1,m])}\!\bigl(e^\theta_{\varphi^{-1}(j)}\bigr)=
\begin{cases}
e^\theta_{\varphi^{-1}(j)}, & \varphi^{-1}(j)\in \varphi^{-1}([1,m])\\
0, & \varphi^{-1}(j)\notin \varphi^{-1}([1,m])
\end{cases}=\begin{cases}
e^\theta_{\varphi^{-1}(j)}, & 1\leq j\leq m,\\
0, & j>m.
\end{cases}
$$
Therefore,
\begin{align*}
T_{\varphi^{-1}([1,m])} \left( \sum\limits_{j=1}^{n}a_j y_j \right)& =T_{\varphi^{-1}([1,m])}\!\left(\sum\limits_{j=1}^n a_j e^\theta_{\varphi^{-1}(j)}\right)\\
&=\sum\limits_{j=1}^m a_j e^\theta_{\varphi^{-1}(j)}
=\sum\limits_{j=1}^m a_j y_j .
\end{align*}
Consequently, by \eqref{T-proj-bdd} we obtain
\begin{align*}
\left\| \sum\limits_{j=1}^{m}a_j y_j \right\|_{L_p(\mathbb{T}_\theta^d)} &= \left\| T_{\varphi^{-1}([1,m])} \left( \sum\limits_{j=1}^{n}a_j y_j \right) \right\|_{L_p(\mathbb{T}_\theta^d)}\\
&\leq \| T_{\varphi^{-1}[1,m]} \|_{L_p(\mathbb{T}_\theta^d)\to L_p(\mathbb{T}_\theta^d)} \cdot \left\| \sum\limits_{j=1}^{n}a_j y_j \right\|_{L_p(\mathbb{T}_\theta^d)}\\
&\leq C_{p,d} \left\| \sum\limits_{j=1}^{n}a_j y_j \right\|_{L_p(\mathbb{T}_\theta^d)}.
\end{align*}
Thus, condition (iii) in Proposition \ref{Schauder} is also satisfied. Therefore, the sequence 
$$
\left\{ y_m = e^{\theta}_{\varphi^{-1}(m)} \right\}_{m=1}^{\infty}
$$
is a Schauder basis in $L_p(\mathbb{T}_\theta^d)$.
This concludes the proof.

\end{proof}
We illustrate our main result with the following example.

\begin{example}Let $d=2.$
For $\gamma=(m,n)\in\mathbb Z^2,$ we write the generalized trigonometric
system as
$$
e^\theta_{(m,n)}:=U_1^{\,m}U_2^{\,n}.
$$
For $r=0,1,2,\dots$ we consider the square 
$$
K_r=\{(m,n)\in\mathbb Z^2:\ |m|\le r,\ |n|\le r\}.
$$
We enumerate $\mathbb Z^2$ by increasing squares: first list the points of $K_0$, then
the new points in $K_1\setminus K_0$, then those in $K_2\setminus K_1$, and so on.
Inside each boundary $K_r\setminus K_{r-1}$ we go around the square in some fixed order
(e.g.\ left side $\to$ right side $\to$ bottom $\to$ top) and along each side we list the
integer points consecutively. This produces a bijection (enumeration)
$$
\varphi:\mathbb Z^2\to\mathbb N.
$$
For example, one may start with
$$
(0,0),\ (-1,-1),\ (-1,0),\ (-1,1),\ (0,1),\ (1,1),\ (1,0),\ (1,-1),\ (0,-1),\dots
$$
Using this enumeration we reorder the generalized trigonometric system into a sequence
$$
y_k:=e^\theta_{\varphi^{-1}(k)},\quad k\ge1.
$$
A similar enumeration of matrix units was already employed by Kwapień and Pełczyński in their study of the main triangle projection in matrix spaces \cite{KP1970}.
Let $1<p<\infty$ and let $x\in\mathcal P_p.$ Then $x$ can be written in the form
$$
x=\sum\limits_{k=1}^N a_k\,y_k,
$$
for some $N\in\mathbb N$ and scalars $a_1,\dots,a_N\in\mathbb C$.
For each $m\le N$ consider the set of the first $m$ indices
$$
A_m:=\varphi^{-1}([1,m])\subset\mathbb Z^2
$$
and the corresponding projection $T_{A_m}$. By definition of $T_A$
(see \eqref{TA}), we have
$$
T_{A_m}(y_k)=
\begin{cases}
y_k, & 1\le k\le m,\\
0,   & k>m,
\end{cases}
$$
and therefore,
$$
T_{A_m}x
=T_{A_m}\Big(\sum\limits_{k=1}^N a_k\,y_k\Big)
=\sum\limits_{k=1}^m a_k\,y_k.
$$
The key point is that each set $A_m=\varphi^{-1}([1,m])$ has a simple shape:
it consists of a whole square $K_{r-1}$ together with a part of the boundary of the next
square $K_r$. Such sets can be constructed from finitely many coordinate half-spaces
(conditions of the form $m\ge -r$, $m\le r$, $n\ge -r$, $n\le r$) by finitely many unions
and intersections. Hence, $A_m$ is of type $(k,l)$ in the sense of Definition~\ref{numeration}.
Consequently, it follows from Lemma~\ref{TAH} that the projections $T_{A_m}$ are uniformly bounded on
$L_p(\mathbb T^2_\theta)$ and we have
$$
\|T_{A_m}\|_{L_p(\mathbb T^2_\theta)\to L_p(\mathbb T^2_\theta)}\le C_p,
\quad m\ge1,
$$
with $C_p>0$ independent of $m$. This uniform boundedness is exactly the condition needed in
Proposition~\ref{Schauder}, and hence the sequence $\{y_k\}_{k\ge1}$ is a Schauder basis in
$L_p(\mathbb T^2_\theta)$.

Finally, when $\theta=0$ the unitaries commute, $\mathbb T^2_\theta$ becomes the classical
 $L_{\infty}(\mathbb T^2)$ and $e^0_{(m,n)}(z_1,z_2)=e^{2\pi imt}e^{2\pi ins}$, so the above example reduces
to the usual ordering of the classical trigonometric system in $L_p(\mathbb T^2)$.
\end{example}

\section{RUC-Bases in noncommutative $L_p(\mathbb{T}^d_\theta)$ spaces}

In this section, we investigate RUC-bases in noncommutative $L_p(\mathbb{T}^d_\theta)$ spaces. We show that the trigonometric system $\{e^{\theta}_{\gamma}\}_{\gamma\in\mathbb Z^d}$ forms an RUC-bases in $L_p(\mathbb{T}^d_\theta)$ for $2<p<\infty$.

\begin{definition}[RUC-basis]\label{def:RUC} Let $X$ be a Banach space and let $\{x_n\}_{n\ge1}$ be a Schauder basis of $X.$ The sequence $\{x_n\}_{n=1}^{\infty}$ is called an \emph{RUC}-basis if for every
$x = \sum\limits_{n=1}^{\infty} \alpha_n x_n \in X,$
the series
$$
\sum\limits_{n=1}^{\infty} r_n(t)\,\alpha_n x_n
$$
converges for almost all $t \in [0,1]$.
\end{definition}

\begin{definition}\cite{LT}\label{def:type}
A Banach space $X$ is said to be of (Rademacher) type $2$ if there exists a constant
$M(X) < \infty$ such that, for every finite family $\{x_k\}_{k=1}^n \subset X$, one has
$$
\left( \int\limits_{0}^{1} \left\| \sum\limits_{k=1}^n r_k(t)\, x_k \right\|_X^{2} \, dt \right)^{1/2}
\le
M(X)\left( \sum\limits_{k=1}^n \|x_k\|_X^{2} \right)^{1/2}.
$$
where $\{r_k\}_{k=1}^{\infty}$ denotes the usual Rademacher sequence given by
$$
r_k(t) = \operatorname{sign}(\sin(2^{k}\pi t)), \quad 0 \le t \le 1 \quad (k = 1,2,\ldots).
$$
\end{definition}
\begin{proposition}\cite[Proposition 1.1]{BKPS}\label{RUC-criteria} A Schauder basis $\{x_n\}_{n\ge1}$ is \emph{RUC}-basis in $X$ if and only if for all $1\leq p<\infty$ there exists constant $K(p,X)>0$ dependent on $p$ and $X$ only such that for any $m\leq n,$ $m,n\in\mathbb N,$ and $\{a_k\}_{k=1}^n\subset \mathbb C,$ we have
\begin{equation}\label{RUC-equ}
\left\| \sum\limits_{k=1}^m r_k(t)a_k x_k \right\|_{L_p\big([0,1],X\big)}
\leq
K(p,X)\left\|\sum\limits_{k=1}^n a_k x_k \right\|_X.
\end{equation}
\end{proposition}
The following is the main result of this section which shows the existence of an RUC-basis in $L_p(\mathbb{T}_\theta^d)$ for $2<p<\infty.$
\begin{theorem}\label{RUC-basis-thm}
Let $\varphi$ be a $(k, l)$-enumeration as in Definition \ref{numeration}. Then the generalized trigonometric system 
$$
\Big\{ y_m = e^{\theta}_{\varphi^{-1}(m)} \Big\}_{m=1}^\infty
$$
forms a \emph{RUC}-basis in  $L_p(\mathbb{T}_\theta^d)$ for $2<p<\infty.$
\end{theorem}
    \begin{proof}Let $2<p<\infty$ and $\varphi$ be the $(k,l)$- enumeration as in Lemma~\ref{phi} (see also
Definition~\ref{numeration}) and set $y_m=e^\theta_{\varphi^{-1}(m)}$.
By Theorem~\ref{main-basis-thm}, $\{y_m\}_{m\ge1}$ is a Schauder basis in
$L_p(\mathbb T_\theta^d).$
Thus, by Proposition \ref{RUC-criteria}, it is sufficient to prove that for any $m\leq n,$ $m,n\in\mathbb N,$ and $\{a_k\}_{k=1}^n\subset \mathbb C,$ we have
\begin{equation}\label{RUC-equ}
\left\| \sum\limits_{k=1}^m r_k(t)a_k y_k \right\|_{L_2\big([0,1],L_p(\mathbb T^d_{\theta})\big)}
\leq
c_p\left\|\sum\limits_{k=1}^n a_k y_k \right\|_{L_p(\mathbb T^d_{\theta})}.
\end{equation}
It is well known (see \cite[Corollary 5.5]{PXu}, \cite[Theorem 7.14.17]{DdPS}) that $L_p(\mathbb T_\theta^d)$ is a Banach space of type $2$ for $2<p<\infty$. Consequently, by \cite[Theorem 1.6]{DS1997} there exist $K_p>0$ such that
\begin{equation}\label{Kahane}
\left\| \sum\limits_{k=1}^m r_k(t)a_k y_k \right\|_{L_2\big([0,1],L_p(\mathbb T^d_{\theta})\big)}
\leq
K_p\left\|\sum\limits_{k=1}^m a_k y_k \right\|_{L_p(\mathbb T^d_{\theta})}.
\end{equation}
Now, let us estimate the RHS. Indeed, as in the proof of Theorem~\ref{main-basis-thm}, we have  
that the projection $T_{\varphi^{-1}([1,m])}$ along the set $\varphi^{-1}([1,m])$ is well defined and bounded on $L_p(\mathbb T^d_{\theta})$ by Lemma~\ref{TAH} with 
$$\|T_{\varphi^{-1}([1,m])}\|_{L_p(\mathbb T^d_{\theta})\to L_p(\mathbb T^d_{\theta})}\leq C_{p,d},\,\ 1<p<\infty.$$ 
Therefore, we have
\begin{align*}
T_{\varphi^{-1}([1,m])} \left( \sum\limits_{j=1}^{n}a_j y_j \right) =T_{\varphi^{-1}([1,m])}\!\left(\sum\limits_{j=1}^n a_j e^\theta_{\varphi^{-1}(j)}\right)
&=\sum\limits_{j=1}^m a_j e^\theta_{\varphi^{-1}(j)}
=\sum\limits_{j=1}^m a_j y_j .
\end{align*}
Consequently, by Lemma~\ref{TAH} we obtain
\begin{align*}
\left\| \sum\limits_{j=1}^{m}a_j y_j \right\|_{L_p(\mathbb{T}_\theta^d)}& = \left\| T_{\varphi^{-1}([1,m])} \left( \sum\limits_{j=1}^{n}a_j y_j \right) \right\|_{L_p(\mathbb{T}_\theta^d)}\\
&\leq \| T_{\varphi^{-1}[1,m]} \|_{L_p(\mathbb{T}_\theta^d)\to L_p(\mathbb{T}_\theta^d)} \cdot \left\| \sum\limits_{j=1}^{n}a_j y_j \right\|_{L_p(\mathbb{T}_\theta^d)} \\
&\leq C_{p,d} \left\| \sum\limits_{j=1}^{n}a_j y_j \right\|_{L_p(\mathbb{T}_\theta^d)}.
\end{align*}
\end{proof}

\begin{rem}
The estimate \eqref{RUC-equ} may also be obtained by applying the contraction principle for Rademacher sums \cite[p. 156]{P}. Moreover, the estimate \eqref{Kahane} can be derived directly from the noncommutative Khintchine inequality, rather than from \cite[Theorem~1.6]{DS1997}. We have chosen the above argument because it uses the uniformly bounded the projections $T_{\varphi^{-1}([1,m])}$, which are naturally associated with the Schauder basis constructed in Theorem~\ref{main-basis-thm}.
\end{rem}

\section{Partial sum operators for trigonometric systems and its weak $(1,1)$ type estimate}

In this section, we study a weak $(1,1)$ type estimate of partial sum operators for trigonometric system.

\begin{lemma}\label{T0}
Let $T_0$ be the projection along $\{0\}\subset\mathbb Z^d$ defined by \eqref{TA}. Then 
$$\tau_{\theta}\bigl(T_0(x)\bigr) = \tau_{\theta}(x), \quad x \in L_1(\mathbb T^d_\theta).$$
\end{lemma}
\begin{proof}
 If $x_\gamma\in L_1(\mathbb T_\theta^d)_\gamma$ for $\gamma\in \mathbb{Z}^d,$ then
$\pi_t(x_\gamma)=t^\gamma x_\gamma$ for all $t\in\mathbb T^d$, so
$$
\int\limits_{\mathbb T^d} \pi_t(x_\gamma)\,dt
 = \left(\int\limits_{\mathbb T^d} t^\gamma\,dt\right)x_\gamma
 =
\begin{cases}
x_\gamma, & \gamma=0,\\
0,        & \gamma\ne 0.
\end{cases}
$$
For $x \in \mathcal{P}_1$, we obtain
$$
\sum\limits_{\gamma\in\mathbb Z^d} \int\limits_{\mathbb T^d} \pi_t(x_\gamma)\,dt = x_0 = T_0(x).
$$
Since $\mathcal{P}_1$ is dense in $L_1(\mathbb T^d_{\theta})$ by Lemma~\ref{Pp-dense}, this identity extends by continuity to 
$ L_1(\mathbb T^d_{\theta})$. Hence,
\begin{equation}\label{eq:T0-average}
T_0(x) = \int\limits_{\mathbb T^d} \pi_t(x)\,dt,
\quad x\in L_1(\mathbb T_\theta^d).
\end{equation}
 We have
$$
\tau_\theta\bigl(T_0(x)\bigr)
 = \tau_\theta\!\left(\int\limits_{\mathbb T^d} \pi_t(x)\,dt\right)
 = \int\limits_{\mathbb T^d} \tau_\theta(\pi_t(x))\,dt
\overset{\text{(\ref{PI}})}{=} \int\limits_{\mathbb T^d} \tau_\theta(x)\,dt
 = \tau_\theta(x), \, \, x\in L_1(\mathbb{T}^d_{\theta}).
$$
This completes the proof.
\end{proof}

We now introduce the following subspace of $L_{\infty}(\mathbb T^{d}_{\theta})$ defined by 
\begin{equation}
\mathcal{A}_{\theta}
 :=
   L_{\infty}(\mathbb T^{d}_{\theta})_{\mathbb Z^{d}_{+}\setminus\{0\}}
 + L_{\infty}(\mathbb T^{d}_{\theta})_{\mathbb Z^{d}_{-}\setminus\{0\}}
 + L_{\infty}(\mathbb T^{d}_{\theta})_{\{0\}}.
\end{equation}
By \eqref{LP}, we have
$$
L_\infty(\mathbb T^d_\theta)
 = L_\infty(\mathbb T^d_\theta)_{\mathbb Z^d}
 = \overline{
      L_\infty(\mathbb T^d_\theta)_{\mathbb Z^d_+\setminus\{0\}}
    + L_\infty(\mathbb T^d_\theta)_{\mathbb Z^d_-\setminus\{0\}}
    + L_\infty(\mathbb T^d_\theta)_{\{0\}}
   }^{\|\cdot\|_{ L_\infty(\mathbb T^d_\theta)}}.
$$
In other words,
$$
\overline{\mathcal A_\theta}^{\|\cdot\|_{ L_\infty(\mathbb T^d_\theta)}}
 = L_\infty(\mathbb T^d_\theta).
$$

For any $x\in \mathcal{A}_{\theta},$ the conjugate of $x$ associated with $\mathbb{Z}^d_+$ is defined by
\begin{equation}
\widetilde{x}
 := i\,x_{-} - i\,x_{+},  
\end{equation}
where $x_{+} \in L_{\infty}(\mathbb T^{d}_{\theta})_{\mathbb Z^{d}_{+}\setminus\{0\}}$ and $x_{-} \in L_{\infty}(\mathbb T^{d}_{\theta})_{\mathbb{Z}^{d}_{-}\setminus\{0\}}.$ 
Observe that
$$
x + i\widetilde{x}
 = x_0 + 2x_{+}
 \in L_{\infty}(\mathbb T^{d}_{\theta})_{\mathbb Z^{d}_{+}},
$$
where $x_0\in L_{\infty}(\mathbb{T}^d_{\theta})_{\{0\}}.$

\begin{definition}\label{HT-def}
The Hilbert transform $\mathcal{H}_{\mathbb{Z}^d_{+}} : \mathcal{A}_{\theta} \to L_{\infty}(\mathbb T^{d}_{\theta})$ associated with $\mathbb{Z}^{d}_{+}$ defined by the formula
\begin{equation}\label{HA}
\mathcal{H}_{\mathbb{Z}^d_{+}}(x)= \widetilde{x}, \,\ x\in \mathcal{A}_{\theta}.
\end{equation}
\end{definition}
Let $G$ be a locally compact abelian group and let $\mathcal M$ be a semifinite von Neumann algebra equipped with a faithful semifinite normal trace $\tau$. In \cite{R2}, Randriantoanina studied Hilbert transforms associated with $G$-flows on $\mathcal M$ and closed semigroups $\Sigma \subset \widehat{G}$ satisfying $\Sigma \cup (-\Sigma)=\widehat{G}$, where $\widehat{G}$ denotes the dual group of $G$.  

The Hilbert transform $\mathcal{H}_{\mathbb{Z}^d_{+}}$ introduced in Definition~\ref{HT-def} coincides with the Hilbert transform defined in \cite[Definition~4.1]{R2} in the special case, when $G=\mathbb T^d$ and $\mathcal M=L_{\infty}(\mathbb T^{d}_{\theta})$. Consequently, Theorem~4.2 and Corollary~4.6 in \cite{R2} yield that $\mathcal{H}_{\mathbb{Z}^d_{+}}$, defined by \eqref{HA}, admits a unique linear bounded extension on $L_p(\mathbb T^d_{\theta})$ for $1<p<\infty$, and from $L_1(\mathbb T^d_{\theta})$ into $L_{1,\infty}(\mathbb T^d_{\theta}).$

Moreover, for $1<p<\infty$ and $x\in L_p(\mathbb T^d_\theta)$ one has
\begin{equation}\label{LpH}
\mathcal{H}_{\mathbb Z^d_+}(x)
 = i\,T_{\mathbb Z^d_{-}\setminus\{0\}}(x)
   - i\,T_{\mathbb Z^d_{+}\setminus\{0\}}(x),
 \quad x\in L_p(\mathbb T^d_\theta),
\end{equation}
where the projections $T_{\mathbb Z^d_{-}\setminus\{0\}}$ and $T_{\mathbb Z^d_{+}\setminus\{0\}}$
are well defined on $L_p(\mathbb T^d_\theta)$ (see Lemma~\ref{lem2}).

\begin{definition}\label{def:Pplus}
The projection associated with $\mathbb Z^d_+$ is the linear map
$P_+:\mathcal A_\theta\to L_\infty(\mathbb T^d_\theta)$ defined by
\begin{equation}\label{Pplus-A}
P_+x
:=\frac12\bigl(x+T_0(x)+ i\,\mathcal H_{\mathbb Z^d_+}(x)\bigr),
\quad x\in\mathcal A_\theta,
\end{equation}
where $T_0$ is the projection along $\{0\}\subset\mathbb Z^d$ defined by \eqref{TA}.
We also define the projection associated with $\mathbb Z^d_-$ by
\begin{equation}\label{def:Pminus}
P_- := I - P_+ + T_0 .
\end{equation}
\end{definition}

Since $\mathcal H_{\mathbb Z^d_+}:L_1(\mathbb T^d_\theta)\to L_{1,\infty}(\mathbb T^d_\theta)$
is bounded \cite[Corollary~4.6]{R2} and, by Lemma~\ref{T0}, the operator $P_+$ defined by
\eqref{Pplus-A} admits a bounded extension
$$
P_+:L_1(\mathbb T^d_\theta)\to L_{1,\infty}(\mathbb T^d_\theta).
$$
Consequently, there exists a constant $C>0$ such that
\begin{equation}\label{Pplus-weak11}
\|P_+x\|_{L_{1,\infty}(\mathbb T^d_\theta)}
\le C\,\|x\|_{L_1(\mathbb T^d_\theta)},
\quad x\in L_1(\mathbb T^d_\theta).
\end{equation}
Similarly, \eqref{def:Pminus} shows that $P_-$ also extends to a bounded map from
$L_1(\mathbb T^d_\theta)$ to $L_{1,\infty}(\mathbb T^d_\theta)$.

Fix a total order $\le_d$ on $\mathbb Z^d$. For $\gamma\in\mathbb Z^d$ define the
right and left closed semi-infinite intervals by
$$
[\gamma,\infty)=\{\eta\in\mathbb Z^d:\ \eta\ge_d \gamma\},
\quad
(-\infty,\gamma]=\{\eta\in\mathbb Z^d:\ \eta\leq_d \gamma\}.
$$
Let
$$
P_{[\gamma,\infty)}:=A_\gamma P_+A_{-\gamma},
\quad
P_{(-\infty,\gamma]}:=A_\gamma P_-A_{-\gamma},
\quad \gamma\in\mathbb Z^d,
$$
where $A_\gamma$ is defined by \eqref{A_gamma}.
For $\mathbf n\in\mathbb Z^d_+$ define the partial sum operators on $\mathcal P_1$ by
\begin{equation}\label{eq:Sn-new}
S_{\mathbf n}x:=x-P_{[\mathbf n,\infty)}x-P_{(-\infty,-\mathbf n]}x=\sum\limits_{\gamma\in F:-\mathbf n\le_d \gamma<_d \mathbf n} c_\gamma\,e^\theta_\gamma,
\quad x\in\mathcal P_1,
\end{equation}
where $F\subset\mathbb Z^d$ is a finite set.

The following result shows the weak (1,1) type estimate of the partial sum operators.
\begin{theorem}\label{prop:Sn-extension-new}
For every $\mathbf n\in\mathbb Z^d_+$, the operator $S_{\mathbf n}$ defined by
\eqref{eq:Sn-new} extends uniquely to a bounded linear map
$$
S_{\mathbf n}:L_1(\mathbb T^d_\theta)\rightarrow L_{1,\infty}(\mathbb T^d_\theta),
$$
and we have
$$
\|S_{\mathbf n}x\|_{L_{1,\infty}(\mathbb T^d_\theta)}
\le C\,\|x\|_{L_1(\mathbb T^d_\theta)},
\quad x\in L_1(\mathbb T^d_\theta),
$$
where $C>0$ is a constant independent of $\mathbf n$.
\end{theorem}
\begin{proof}
Since $\mathcal P_1$ is dense in $L_1(\mathbb T^d_\theta)$ by Lemma~\ref{Pp-dense}, it is enough to prove the result  on
$\mathcal P_1.$ Fix $\mathbf n\in\mathbb Z^d_+$. By \eqref{Pplus-weak11} and the corresponding
estimate for $P_-$, together with Lemma~\ref{A_gamma1}, the projections $P_{[\gamma,\infty)}$ and
$P_{(-\infty,\gamma]}$, defined above, are bounded from
$L_1(\mathbb T^d_\theta)$ to $L_{1,\infty}(\mathbb T^d_\theta)$, with constants
independent of $\gamma$. Therefore, by formula \eqref{eq:Sn-new} and the embedding
$L_1\hookrightarrow L_{1,\infty}$, we obtain
$$
\|S_{\mathbf n}x\|_{L_{1,\infty}(\mathbb T^d_\theta)}
\le C\|x\|_{L_1(\mathbb T^d_\theta)},
\qquad x\in\mathcal P_1,
$$
where $C>0$ is independent of $\mathbf n,$ thereby completing the proof.
\end{proof}

\section{Boundedness of the Hilbert transform and partial sum operators from one symmetric space $\mathcal{E}(\mathbb T^d_{\theta})$ to another $\mathcal{F}(\mathbb T^d_{\theta})$}

In this section, we establish the uniform boundedness of the partial sum operators and the associated Hilbert transform between noncommutative quasi-Banach symmetric spaces, extending the results in \cite{DDdPS,DdPS,R,R2} to the setting of quasi-Banach symmetric spaces whose Boyd indices are not necessarily nontrivial.

For each $f\in L_1(0,1)$, define the Calderón operator $S:L_1(0,1)\to L_{1,\infty}(0,1)$ by
\begin{equation}\label{eq_def_Cald}
(Sf)(t):=\frac{1}{t}\int\limits_0^t f(s)\,ds+\int\limits_t^1\frac{f(s)}{s}\,ds.
\end{equation}
Let $E(0,1)\subset L_1(0,1)$ be a symmetric quasi-Banach space. Define
\begin{equation}\label{opt-space}
\mathcal{S}_E(0,1)=\{f\in L_{1,\infty}(0,1): \exists\, g\in E(0,1)\ \text{such that}\ \mu(f)\le S\mu(g)\},
\end{equation}
and equip this space with the functional
$$
\|f\|_{\mathcal{S}_E (0,1)}:=\inf\{\|g\|_{E(0,1)}:\mu(f)\le S\mu(g)\}.
$$
The space $\mathcal{S}_E$ on $(0,\infty)$ was fisrt introduced in \cite{STZ1} (see also \cite[Definition~15]{AHPSZ} for the case $(0,1)$). It was shown in \cite[Theorem~26]{STZ1} (see also \cite[Theorem~16]{AHPSZ}) that $(\mathcal{S}_E,\|\cdot\|_{\mathcal{S}_E})$ is a symmetric quasi-Banach function space on $(0,1)$ and is the least space for which the Calderón operator $S$ acts boundedly from $E(0,1)$ into $\mathcal{S}_E(0,1)$.

The following theorem is the main result of this section. It establishes the uniform boundedness of the partial sum operators introduced in \eqref{eq:Sn-new} from $\mathcal{E}(\mathbb T^d_{\theta})$ into $\mathcal{S}_{\mathcal{E}}(\mathbb T^d_{\theta})$. Moreover, it extends \cite[Theorem~4.2]{R2} in the special case, when $\mathcal M=L_{\infty}(\mathbb T^d_{\theta})$.

\begin{theorem}\label{EtoF-thm}
Let $E(0,1)\subset L_1(0,1)$ be a symmetric quasi-Banach space and let $\mathcal{S}_E(0,1)$ be the symmetric quasi-Banach space defined by \eqref{opt-space}. Let $\mathcal{E}(\mathbb T^d_{\theta})$ and $\mathcal{S}_{\mathcal{E}}(\mathbb T^d_{\theta})$ be the associated noncommutative symmetric spaces. Then, partial sum operators $S_{\mathbf n}$ and the Hilbert transform $\mathcal{H}_{\mathbb{Z}^d_{+}}$, defined by \eqref{eq:Sn-new} and \eqref{HA}, respectively, admit unique linear bounded extensions from $\mathcal{E}(\mathbb T^d_{\theta})$ into $\mathcal{S}_{\mathcal{E}}(\mathbb T^d_{\theta})$.
\end{theorem}

\begin{proof}
By Theorems~\ref{main-basis-thm} and \ref{prop:Sn-extension-new}, the partial sum operators $S_{\mathbf n}$ extends to a uniformly bounded linear operator on $L_p(\mathbb T^d_{\theta})$ for $1<p<\infty$, and from $L_1(\mathbb T^d_{\theta})$ into $L_{1,\infty}(\mathbb T^d_{\theta})$. Similarly, by Theorem~4.2 and Corollary~4.6 in \cite{R2}, the Hilbert transform $\mathcal{H}_{\mathbb{Z}^d_{+}}$ admits bounded linear extensions on $L_p(\mathbb T^d_{\theta})$ for $1<p<\infty$, and from $L_1(\mathbb T^d_{\theta})$ into $L_{1,\infty}(\mathbb T^d_{\theta})$.
Therefore, both operators satisfy the hypotheses of \cite[Theorem~14]{STZ1}, and consequently,
$$
\mu(S_{\mathbf n}(x))\lesssim S\mu(x), \quad x\in L_1(\mathbb T^d_{\theta}),
$$
and
$$
\mu(\mathcal{H}_{\mathbb{Z}^d_{+}}(y))\lesssim S\mu(y), \quad y\in L_1(\mathbb T^d_{\theta}).
$$
It follows that
\begin{align*}
\|S_{\mathbf n}(x)\|_{\mathcal{S}_{\mathcal{E}}(\mathbb T^d_{\theta})}
&=\|\mu(S_{\mathbf n}(x))\|_{\mathcal{S}_E(0,1)}
\lesssim \|S\|_{E(0,1)\to\mathcal{S}_E(0,1)}\,\|\mu(x)\|_{E(0,1)} \\
&=\|S\|_{E(0,1)\to\mathcal{S}_E(0,1)}\,\|x\|_{\mathcal{E}(\mathbb T^d_{\theta})},
\quad x\in\mathcal{E}(\mathbb T^d_{\theta}).
\end{align*}
An analogous estimate holds for $\mathcal{H}_{\mathbb{Z}^d_{+}}$, that is,
$$
\|\mathcal{H}_{\mathbb{Z}^d_{+}}(y)\|_{\mathcal{S}_{\mathcal{E}}(\mathbb T^d_{\theta})}
\lesssim
\|S\|_{E(0,1)\to\mathcal{S}_E(0,1)}\,\|y\|_{\mathcal{E}(\mathbb T^d_{\theta})},
\quad y\in\mathcal{E}(\mathbb T^d_{\theta}),
$$
which completes the proof.
\end{proof}

If $E(0,1)=L_1(0,1),$ then it was proved in \cite[Proposition 35]{STZ1} that $\mathcal{S}_{L_1}(0,1)=(L_{1,\infty}(0,1))^0,$ where
$$(L_{1,\infty}(0,1))^0:=\{f\in L_{1,\infty}(0,1):t\mu(t,f)\to 0, \,\,\text{as}\,\ t\to0\}.$$
Define 
$$(L_{1,\infty}(\mathbb T^d_{\theta}))^0:=\{x\in L_{1,\infty}(\mathbb T^d_{\theta})):\mu(x)\in (L_{1,\infty}(0,1))^0 \}.$$
Then, the following refines the results in \cite[Corollary~4.6]{R2} and \cite[Theorem 1.4]{DDdPS2}.
\begin{corollary}The partial sum operators $S_{\mathbf n}$ and the Hilbert transform $\mathcal{H}_{\mathbb{Z}^d_{+}}$ defined by \eqref{eq:Sn-new} and \eqref{HA}, respectively, 
$$S_{\mathbf n}, \mathcal{H}_{\mathbb{Z}^d_{+}}:L_1(\mathbb T^d_{\theta})\to (L_{1,\infty}(\mathbb T^d_{\theta}))^0$$ are bounded.
    
\end{corollary}

\subsection{Conflict of Interest Statement:} The author declares that there is no conflict of interest.

\subsection{ Data Availability Statement:} No datasets were generated or analyzed during the current study.
    
\section{Acknowledgements}
%The work was partially supported by the grant No. AP23487088 of the Science Committee of the Ministry of Science and Higher
%Education of the Republic of Kazakhstan.

The work was partially supported by the Australian Research Council. K.T. was partially supported by ARC grant FL17010005. F.S. was supported by ARC grant FL17010005 and ARC grant DP230100434. 
%Authors thank the anonymous referees for reading the paper and providing thoughtful comments, which improved the exposition of the paper.

\end{document}